\documentclass[a4paper,11pt,oneside]{article}

\usepackage{amsmath,amsthm,amsfonts,amssymb,amscd}
\usepackage{mathtools}
\usepackage{pgf,tikz}
\usepackage{mathrsfs}
\usepackage{float}
\usepackage{csvsimple}
\usepackage{enumerate}
\usepackage{booktabs}
\usepackage{pgfplots}
\usepackage[noend]{algpseudocode}
\usepackage{graphicx,bm,xcolor}
\usepackage{algorithm}
\usepackage{pdfpages}
\usepackage{algpseudocode}
\usepackage{stmaryrd}
\usepackage[square,sort,comma,numbers]{natbib}
\usetikzlibrary{arrows}
\usepackage{caption}

\usepackage{t1enc}
\usepackage[german,english]{babel}
\usepackage{verbatim} 

\usepackage{url}

\newif\ifMAKEPICS
\MAKEPICSfalse

\ifMAKEPICS
\usepackage[cleanup,subfolder]{gnuplottex}
\usepackage{xparse}

\ExplSyntaxOn
\DeclareExpandableDocumentCommand{\convertlen}{ O{cm} m }
{
	\dim_to_decimal_in_unit:nn { #2 } { 1 #1 } cm
}
\ExplSyntaxOff
\fi

\usepackage{authblk}

\usepackage[top=2cm, bottom=2cm, left=2cm, right=2cm]{geometry}

\newtheorem{theorem}{Theorem}[section]
\newtheorem{proposition}{Proposition}[section]
\newtheorem{lemma}[theorem]{Lemma}

\newtheorem{remark}[theorem]{Remark}

\newtheorem{assumption}{Assumption}

\newcommand{\red}[1]{{\color{red}{#1}}}
\newcommand{\ignore}[1]{}

\begin{document}
	\title{Reliability and efficiency  of DWR-type  a posteriori error estimates with smart sensitivity weight recovering}	
	\author[1,2]{B. Endtmayer}
	\author[2]{U. Langer}
	\author[3,4]{T. Wick}
	
	\affil[1]{Doctoral Program on Computational Mathematics, Johannes Kepler University, Altenbergerstr. 69, A-4040 Linz, Austria}
	\affil[2]{Johann Radon Institute for Computational and Applied Mathematics, Austrian Academy of Sciences, Altenbergerstr. 69, A-4040 Linz, Austria}
	
	\affil[3]{Leibniz Universit\"at Hannover,
		Institut f\"ur Angewandte Mathematik,
		AG Wissenschaftliches Rechnen,
		Welfengarten 1, 30167 Hannover, Germany}
	\affil[4]{Cluster of Excellence PhoenixD (Photonics, Optics, and
		Engineering - Innovation Across Disciplines), Leibniz Universit\"at Hannover, Germany}
	
	\date{}
	
	\maketitle

	\begin{abstract}
		We derive efficient and reliable goal-oriented error estimations,
		and devise adaptive mesh procedures for the finite element method  
		that are based on the localization of a posteriori estimates.
		In our previous work [SIAM J. Sci. Comput., 42(1), A371--A394, 2020], 
		we showed efficiency and reliability for 
		error estimators based on enriched finite element spaces.
		However, the solution of problems on a enriched finite element space is 
		expensive.
		In the literature, it is well known that one can use some higher-order 
		interpolation to overcome this 
		bottleneck. 
		Using a saturation assumption,  
		we extend the proofs of efficiency and reliability to 
		such higher-order interpolations.
		The results can be used to create a new family of algorithms,
		where one of them is tested on three numerical examples 
		(Poisson problem, p-Laplace equation, Navier-Stokes benschmark),
		and is compared to our previous algorithm.
	\end{abstract}

	\section{Introduction}
	\label{Section: Introduction}
	
	Goal-oriented error estimation using adjoints 
	was established in \cite{BeckerRannacher1995,BeRa01},
	and is a current research topic as attested in many studies. Recent 
	developments include 
	Fractional-Step-$\theta$ Galerkin formulations \cite{MeiRi14}, 
	a partition-of-unity variational 
	localization \cite{RiWi15_dwr},
	phase-field fracture problems \cite{Wi16_dwr_pff},
	surrogate constructions for stochastic 
	inversion \cite{MATTIS201836}, general adaptive multiscale modeling 
	\cite{oden_2018}, model adaptivity in multiscale problems \cite{MaiRa18},
	multigoal-oriented error estimation with balancing discretization and iteration
	errors \cite{EnLaWi18},
	nonstationary, nonlinear fluid-structure interaction \cite{FaiWi18}, 
	realizations on polygonal meshes using boundary-element based finite elements \cite{WeiWi18}, 
	realizations using the finite cell method \cite{StRaSchroe19},
	error estimation for sea ice simulations \cite{MeRi2020},
	discretization error estimation in computer assisted surgery \cite{DuBoBuBuChLlLoLoRoTo2020},
	and a unifying framework with inexact solvers covering discontinuous Galerkin and finite volume approaches \cite{MaVohYou20}.
	An open-source framework for linear, time-dependent, goal-oriented 
	error estimation was published in \cite{KoeBruBau2019a}.
	Abstract convergence results of goal-oriented techniques were 
	studied in \cite{HolPo2016} and \cite{FeiPraeZee16}. A 
	worst-case multigoal-oriented 
	error estimation was carried out in \cite{BruZhuZwie16}.
	Recently, using a saturation assumption, e.g., 
	\cite{DoerflerNochetto2002,CaGaGed16,FeiPraeZee16,BankParsaniaSauter2013},
	two-side error estimates, namely efficiency and robustness
	of the adjoint-based error estimator,
	could be shown \cite{EnLaWi20}.

	In most realizations, an adjoint problem is used in order 
	to determine sensitivities that enter the error in a single 
	target functional \cite{BeRa01,RanVi2013} or multiple quantities 
	of interest \cite{Ha08,EnWi17,EnLaWi18}. In \cite{RanVi2013,EnLaWi18},
	balancing of discretization and nonlinear iteration errors 
	for single and multiple goal functionals was considered, respectively.
	However, for nonlinear problems, sensitivity weights in the primal 
	and adjoint variable are necessary. 
	These weights must be of higher-order. Otherwise they yield zero sensitivities,
	and, therefore, the entire error estimator vanishes.
	%
	The most straightforward way is to use 
	higher-order finite element spaces \cite{BeRa01}. 
	A computationally cheaper approach based on low-order finite elements
	and then a patch-wise
	higher-order interpolation was suggested 
	as well in the early stages \cite{BeRa01} and combined 
	in an elegant way to a weak localization in \cite{BraackErn02}. 
	Rigorous proofs of effectivity measured in mesh-dependent norms, 
	in which the true error and the 
	estimator satisfy a common upper bound, were established in \cite{RiWi15_dwr}.
	Therein, saturation assumptions where not required, but the results 
	of effectivity are weaker than in our recent work \cite{EnLaWi20}.
	More specifically, in \cite{EnLaWi20}, 
	the adjoint problem is solved in a higher-order space and then 
	interpolated into the low-order space for calculating the 
	interpolation error. For this procedure, two-side error estimates 
	are proven, while using the previously mentioned saturation assumption.

	The objective of this paper is now the other way around, namely
	efficiency proofs (under saturation assumptions)
	for interpolations from low-order finite element spaces
	into higher-order spaces, which are not present in the literature to date.
	We establish results 
	when both the primal and adjoint problems are computed 
	only in low-order finite element spaces. To this end, 
	two additional error terms will be introduced. The resulting 
	error estimator holds for nonlinear PDEs and nonlinear 
	goal functionals, and accounts for the discretization error,
	nonlinear iteration error and the interpolation error.
	These derivations 
	lead to
	a novel adaptive 
	algorithm that works at low computational cost in goal-oriented 
	frameworks. Our theoretical 
	and algorithmic developments are then substantiated 
	with the help of several numerical tests.
	It is known that the saturation assumption may be violated for 
	transport or convection-dominated problems. We provide 
	one numerical test in which this is the case. 
	Furthermore, linear and nonlinear problem configurations 
	are considered to cover most relevant classes of stationary,
	nonlinear settings.
	
	The outline of this paper is as follows: In Section 
	\ref{Section: The dual weighted residual method for nonlinear problems },
	an abstract setting for the DWR method is introduced. Then, 
	in Section 
	\ref{Section: Efficiency and reliability results for the DWR estimator in enriched spaces},
	the DWR approach with general approximations is stated.
	A discussion of all terms of the newly proposed error estimator 
	is provided in Section 
	\ref{Section: Localization and discussions on the error estimator parts}.
	Based on these findings, an adaptive algorithm is designed 
	in Section \ref{Section: Algorithms}. Several numerical 
	tests with linear PDEs and linear goal functionals, nonlinear 
	settings, and a stationary Navier-Stokes configuration  are carried out in  Section 
	\ref{Section: Numerical examples}. We summarize our results 
	in Section \ref{Section: Conclusions}.

	\section{An abstract setting and the dual weighted resiudal method} 	
	\label{Section: The dual weighted residual method for nonlinear problems }
	In this section, we briefly introduce the notation and the settings 
	that we consider in this work. These are similar to our 
	previous studies \cite{EnLaWi18,EnLaWi20}.

	\subsection{An abstract setting}
	Let $U$ and $V$ be reflexive Banach spaces, and let $\mathcal{A}: U \mapsto V^*$ be a nonlinear mapping, 
	where $V^*$ denotes the dual space of the Banach space $V$.
	We consider the problem: Find $u \in U$ such that
	\begin{equation} \label{Equation: Cont Primal Problem}
	\mathcal{A}(u)=0 \qquad \text{ in } V^*.
	\end{equation}
	This problem will be refereed as the \textit{primal problem}. As mapping we have in mind some (possibly) nonlinear partial differential operator. 
	Furthermore, we consider finite dimensional subspaces of $U_h \subset U$ and $V_h \subset V$.
	In this work, $U_h$ and $V_h$ are finite element spaces .
	This leads to the following finite dimensional problem: Find $u_h \in U_h$ such that
	\begin{equation} \label{Equation: Discret Primal Problem}
	\mathcal{A}(u_h)=0 \qquad \text{ in } V_h^*.
	\end{equation}
	We assume that the problem (\ref{Equation: Cont Primal Problem}) as well as the finite dimensional problem (\ref{Equation: Discret Primal Problem}) are solvable. Further conditions will be imposed when needed.
	However, we do not aim for the solution $u$. The goal is to obtain one characteristic quantity (quantity of interest) $J(u)\in \mathbb{R}$, i.e., a functional evaluation, evaluated at the solution $u \in U$ , where $J: U \mapsto \mathbb{R}$.

	\subsection{The dual weighted residual method}
	\label{PU-DWR-NONLINEAR-ONEFUNTIONAL}
	In this section, we briefly review the Dual Weighted Residual (DWR) method for nonlinear problems.
	This work was extended to balance the discretization and iteration errors in \cite{RanVi2013,RaWeWo10,MeiRaVih109}. 
	
	This paper forms together with our previous works \cite{EnLaWi18,EnLaWi20}
	the basis of the current
	study.  
	Since the DWR method is an adjoint based method, we consider 
	the  \textit{adjoint problem}: Find $z \in V$ such that 
	\begin{equation}\label{Equation : cont adjoint Problem}
	\left(\mathcal{A}'(u)\right)^* (z) = J'(u) \qquad \text{ in } U^*,
	\end{equation}
	where $ \mathcal{A}'(u)$ and $J'(u)$ are the Fr\'echet-derivatives of the nonlinear operator and functional, respectively, evaluated at the solution of the primal problem $u$.
	In the following sections, we require a  finite dimensional version of 
	(\ref{Equation : cont adjoint Problem}) that reads as follows:
	Find $z_h \in V_h$ such that 
	\begin{equation*}\label{Equation : discrete adjoint Problem}
	\left(\mathcal{A}'(u_h)\right)^* (z_h) = J'(u_h) \qquad \text{ in } U_h^*.
	\end{equation*}
	Similarly to  the findings in \cite{RanVi2013,BeRa01,RaWeWo10,EnLaWi20},  
	we obtain   an error representation in the following theorem: 
	\begin{theorem}\label{Theorem: Error Representation}
		Let us assume that $\mathcal{A} \in \mathcal{C}^3(U,V)$ and $J \in \mathcal{C}^3(U,\mathbb{R})$. 
		If $u$ solves (\ref{Equation: Cont Primal Problem}) 
		and $z$ solves (\ref{Equation : cont adjoint Problem}) for $u \in U$, 
		then 
		the error representation
		\begin{align} \label{Error Representation}
		\begin{split}
		J(u)-J(\tilde{u})&= \frac{1}{2}\rho(\tilde{u})(z-\tilde{z})+\frac{1}{2}\rho^*(\tilde{u},\tilde{z})(u-\tilde{u}) 
		-\rho (\tilde{u})(\tilde{z}) + \mathcal{R}^{(3)},\nonumber
		\end{split}
		\end{align}
		holds true for arbitrary fixed  $\tilde{u} \in U$ and $ \tilde{z} \in V$, where
		$\rho(\tilde{u})(\cdot) := -\mathcal{A}(\tilde{u})(\cdot)$,
		$\rho^*(\tilde{u},\tilde{z})(\cdot) := J'(u)-\mathcal{A}'(\tilde{u})(\cdot,\tilde{z})$,
		and the remainder term
		\begin{equation}
		\begin{split}	\label{Error Estimator: Remainderterm}
		\mathcal{R}^{(3)}:=\frac{1}{2}\int_{0}^{1}\big[J'''(\tilde{u}+se)(e,e,e)
		-\mathcal{A}'''(\tilde{u}+se)(e,e,e,\tilde{z}+se^*)
		-3\mathcal{A}''(\tilde{u}+se)(e,e,e)\big]s(s-1)\,ds,
		\end{split} 
		\end{equation}
		with $e=u-\tilde{u}$ and $e^* =z-\tilde{z}$.
	\end{theorem}
	\begin{proof}
		The proof is given in \cite{EnLaWi18}.
	\end{proof}
	
	For the arbitrary elements $\tilde{u}$ and $\tilde{z}$, we think of approximations to the discrete solutions $u_h$ and $z_h$.
	The resulting error estimator reads as 
	\begin{equation*} \label{FullErrorEstimator}
	\eta=\frac{1}{2}\rho(\tilde{u})(z-\tilde{z})+\frac{1}{2}\rho^*(\tilde{u},\tilde{z})(u-\tilde{u}) 
	+\rho (\tilde{u})(\tilde{z}) + \mathcal{R}^{(3)}.
	\end{equation*} 
	This error estimator is exact, but it still depends on the unknown solutions $u$ and $z$.  
	Therefore, it is not computable.		
	To obtain a computable error estimator,  one can replace the exact solutions  $u$ and $z$
	by an approximate solution  on enriched finite dimensional spaces $U_h^{(2)}$ and
	$V_h^{(2)}$.  
	This was already discussed in 
	\cite{BeRa01,BaRa03}.
	Some efficiency and reliability results for this replacement are discussed in \cite{EnLaWi20}.  Other replacements will be covered by the theory in this work. This includes (patch-wise) reconstructions as suggested in \cite{BeRa01,BaRa03,BraackErn02} .

	
	\section{DWR error estimation using general approximations}
	\label{Section: Efficiency and reliability results for the DWR estimator in enriched spaces}
	In this key section, the DWR estimator is augmented 
	to deal with general approximations such as interpolations.
	The latter allow for a very cost-efficient realization
	of the DWR estimator for both linear and nonlinear problems.
	These improvements are significant. 
	However, it turns out 
	that the governing proofs have the same structure as 
	in Section 3 in  our previous work \cite{EnLaWi20}. Therein,
	it was assumed that the solutions on  enriched spaces are known.

	In the current work, we consider that we just have some arbitrary approximations in the enriched spaces.  
	Examples how this is done are by inaccurate  or accurate solves on the enriched space or (patch-wise) higher-order interpolation operators \cite{BeRa01,BaRa03}.
	We show efficiency and reliability for an alternate form of the error estimator using some saturation assumption for the goal functional  on two different kind of approximations. 
	In other words, this means that the approximations in the enriched spaces deliver a more accurate result in the quantity of interest than the approximation of $u_h$.

	\subsection{Results on discrete spaces}	

	Let $\tilde{u}_h^{(2)} \in U_h^{(2)}$ be some arbitrary but fixed approximation of the primal problem: Find ${u}_h^{(2)} \in U_h^{(2)}$ such that 
	$\mathcal{A}({u}_h^{(2)})=0$ in $(V_h^{(2)})^* $, and $\tilde{z}_h^{(2)} \in V_h^{(2)}$ be an approximation to the discretized adjoint problem : Find ${z}_h^{(2)} \in V_h^{(2)}$
	$(\mathcal{A'}({u}_h^{(2)}))^*({z}_h^{(2)}) =J'({u}_h^{(2)})$ in $(U_h^{(2)})^*.$
	%
	\begin{theorem}\label{Corrollary: discrete Error Representation}
		Let us assume that $\mathcal{A} \in \mathcal{C}^3(U,V)$ and $J \in \mathcal{C}^3(U,\mathbb{R})$, and let   $\tilde{u} \in U$ and $ \tilde{z} \in V$ be arbitrary but fixed. 
		If $\tilde{u}_h^{(2)} \in U_h^{(2)}$ and $\tilde{z}_h^{(2)} \in V_h^{(2)}$ are some approximations  of  $u$ and $z$,
		then 
		the error representation
		\begin{align*} 
		\begin{split}
		J(\tilde{u}_h^{(2)})-J(\tilde{u})&= \frac{1}{2}\rho(\tilde{u})(\tilde{z}_h^{(2)}-\tilde{z})+\frac{1}{2}\rho^*(\tilde{u},\tilde{z})(\tilde{u}_h^{(2)}-\tilde{u}) 
		-\rho (\tilde{u})(\tilde{z}) \\ &-\rho(\tilde{u}_h^{(2)})(\frac{\tilde{z}_h^{(2)}+ \tilde{z}}{2})+\frac{1}{2}\rho^*(\tilde{u}_h^{(2)},\tilde{z}_h^{(2)})(\tilde{u}_h^{(2)}-\tilde{u})+ \tilde{\mathcal{R}}^{(3)(2)}
		\end{split}
		\end{align*}
		holds.
		In this error representation, the new terms in comparison to \cite{EnLaWi20} 
		are 
		\[
		\rho(\tilde{u}_h^{(2)})(\frac{\tilde{z}_h^{(2)} +\tilde{z}}{2})
		\quad\text{and}\quad
		\frac{1}{2}\rho^*(\tilde{u}_h^{(2)},\tilde{z}_h^{(2)})(\tilde{u}_h^{(2)}-\tilde{u}).
		\]
		Furthermore, we have 
		$\rho(\tilde{u})(\cdot) := -\mathcal{A}(\tilde{u})(\cdot)$
		and
		$\rho^*(\tilde{u},\tilde{z})(\cdot) := J'(u)-\mathcal{A}'(\tilde{u})(\cdot,\tilde{z})$
		as usual.
		%
		Finally, the remainder term is given by
		\begin{equation*}
		\begin{split}	\label{Error Estimator: Remainderterm1}
		\tilde{\mathcal{R}}^{(3)(2)}:=\frac{1}{2}\int_{0}^{1}[J'''(\tilde{u}+se)(e,e,e)
		-\mathcal{A}'''(\tilde{u}+se)(e,e,e,\tilde{z}+se^*)
		-3\mathcal{A}''(\tilde{u}+se)(e,e,e)]s(s-1)\,ds,
		\end{split} 
		\end{equation*}
		with $e=\tilde{u}_h^{(2)}-\tilde{u}$ and $e^* =\tilde{z}_h^{(2)}-\tilde{z}$.

		\begin{proof}
			The proof is an extension of \cite{RanVi2013}.
			First we define a general approximation $x := (\tilde{u}_h^{(2)},\tilde{z}_h^{(2)}) \in  X:=U \times V$ and $\tilde{x}:=(\tilde{u},\tilde{z}) \in X$.
			Assuming
			that $\mathcal{A} \in \mathcal{C}^3(U,V)$ and $J \in
			\mathcal{C}^3(U,\mathbb{R})$, we know that the Lagrange functional,  which is given by
			\begin{equation*}
			\mathcal{L}(\hat{x}):= J(\hat{u})-\mathcal{A}(\hat{u})(\hat{z}) \quad \forall (\hat{u},\hat{z})=:\hat{x} \in X,
			\end{equation*}
			belongs to
			$\mathcal{C}^3(X,\mathbb{R})$. This allows us to derive the following identity
			\begin{equation*}
			\mathcal{L}(x)-\mathcal{L}(\tilde{x})=\int_{0}^{1} \mathcal{L}'(\tilde{x}+s(x-\tilde{x}))(x-\tilde{x})\,ds.
			\end{equation*}
			Using  the trapezoidal rule 
			\begin{equation*}
			\int_{0}^{1}f(s)\,ds =\frac{1}{2}(f(0)+f(1))+ \frac{1}{2} \int_{0}^{1}f''(s)s(s-1)\,ds,
			\end{equation*}
			with $f(s):= \mathcal{L}'(\tilde{x}+s(x-\tilde{x}))(x-\tilde{x})$, cf. \cite{RanVi2013}, we obtain
			\begin{align*}
			\mathcal{L}(x)-\mathcal{L}(\tilde{x}) =& \frac{1}{2}(\mathcal{L}'(x)(x-\tilde{x}) +\mathcal{L}'(\tilde{x})(x-\tilde{x})) + \mathcal{R}^{(3)}.
			\end{align*}
			From the definition of $\mathcal{L}$,  we observe that
			\begin{align*}
			J(\tilde{u}_h^{(2)})-J(\tilde{u})=&\mathcal{L}(x)-\mathcal{L}(\tilde{x}) +{A(\tilde{u}_h^{(2)})(\tilde{z}_h^{(2)}) }- A(\tilde{u})(\tilde{z}) \\=& \frac{1}{2}(\mathcal{L}'(x)(x-\tilde{x}) +\mathcal{L}'(\tilde{x})(x-\tilde{x})) +A(\tilde{u}_h^{(2)})(\tilde{z}_h^{(2)}) -A(\tilde{u})(\tilde{z})+ \tilde{\mathcal{R}}^{(3)(2)}.
			\end{align*}
			It holds
			\begin{align*}
			\mathcal{L}'(x)(x-\tilde{x}) +\mathcal{L}'(\tilde{x})(x-\tilde{x}) = & \underbrace{J'(\tilde{u}_h^{(2)})(e)-\mathcal{A}'(\tilde{u}_h^{(2)})(e,\tilde{z}_h^{(2)})}_{=\rho^*(\tilde{u}_h^{(2)},\tilde{z}_h^{(2)})(\tilde{u}_h^{(2)}-\tilde{u})}-{A(\tilde{u}_h^{(2)})(e^*)}\\+ &\underbrace{J'(\tilde{u})(e)-\mathcal{A}'(\tilde{u})(e,\tilde{z})}_{=\rho^*(\tilde{u},\tilde{z})(\tilde{u}_h^{(2)}-\tilde{u})}-\underbrace{A(\tilde{u})(e^*)}_{=\rho(\tilde{u})(\tilde{z}_h^{(2)}-\tilde{z})}.
			\end{align*}
			Further manipulations and rewriting, together with 
			\[
			-\frac{1}{2}A(\tilde{u}_h^{(2)})(e^*) + {A(\tilde{u}_h^{(2)})(\tilde{z}_h^{(2)})}=
			A(\tilde{u}_h^{(2)})(-\frac{1}{2}\tilde{z}_h^{(2)} +\frac{1}{2}\tilde{z}+\tilde{z}_h^{(2)}) ) 
			= {A(\tilde{u}_h^{(2)})(\frac{\tilde{z}_h^{(2)} +\tilde{z}}{2}) },
			\] 
			yield
			\begin{align*}
			J(\tilde{u}_h^{(2)})-J(\tilde{u})=\frac{1}{2}\big(\rho^*(\tilde{u}_h^{(2)},\tilde{z}_h^{(2)})(\tilde{u}_h^{(2)}-\tilde{u})
			-A(\tilde{u}_h^{(2)})(e^*)
			+\rho(\tilde{u})(\tilde{z}_h^{(2)}-\tilde{z})+\rho^*(\tilde{u},\tilde{z})(\tilde{u}_h^{(2)}- \tilde{u})\big)\\%
			+{A(\tilde{u}_h^{(2)})(\tilde{z}_h^{(2)}) }) -A(\tilde{u})(\tilde{z})+ \tilde{\mathcal{R}}^{(3)(2)}\\
			=\frac{1}{2}\big({{\rho^*(\tilde{u}_h^{(2)},\tilde{z}_h^{(2)})(\tilde{u}_h^{(2)}-\tilde{u})}+\rho(\tilde{u})(\tilde{z}_h^{(2)}-\tilde{z})+\rho^*(\tilde{u},\tilde{z})(\tilde{u}_h^{(2)}-\tilde{u})}\big) \\
			+\underbrace{{A(\tilde{u}_h^{(2)})(\frac{\tilde{z}_h^{(2)} +\tilde{z}}{2}) })}_{=-\rho(\tilde{u}_h^{(2)})(\frac{\tilde{z}_h^{(2)}+ \tilde{z}}{2})} +A(\tilde{u})(\tilde{z})+ \tilde{\mathcal{R}}^{(3)(2)}.
			\end{align*}
			These last statements 
			prove
			the assertion.
		\end{proof}
		
	\end{theorem}

	{\begin{remark}
			If we use the solutions of the adjoint and the primal problem on the enriched spaces 
			for computing  $\tilde{z}_h^{(2)}$ and $\tilde{u}_h^{(2)}$, respectively, 
			then 
			\[
			\rho(\tilde{u}_h^{(2)})(\frac{\tilde{z}_h^{(2)}+ \tilde{z}}{2}) =0
			\quad \mbox{ and }  \quad
			\frac{1}{2}\rho^*(\tilde{u}_h^{(2)},\tilde{z}_h^{(2)})(\tilde{u}_h^{(2)}-\tilde{u})=0.
			\]
		\end{remark}}
		
		{\begin{remark}
				A practial choice for $\tilde{z}_h^{(2)}$ and $\tilde{u}_h^{(2)}$ is, that we can construct 
				higher-order interpolations from $\tilde{z}_h$ and $\tilde{u}_h$.
				This allows to compute all subproblems with low-order finite elements
				and the error estimator can be constructed in a cheap way. 
				Under a saturation assumption we  prove some results for this well known interpolation techniques.
			\end{remark}}
			
			Theorem~\ref{Corrollary: discrete Error Representation} motivates the following choice of the error estimator.
			\begin{equation} \label{Definition: Error Estimator}
			\tilde{\eta}^{(2)}:=\frac{1}{2}\rho(\tilde{u})(\tilde{z}_h^{(2)}-\tilde{z})+\frac{1}{2}\rho^*(\tilde{u},\tilde{z})(\tilde{u}_h^{(2)}-\tilde{u}) 
			-\rho (\tilde{u})(\tilde{z})-\rho(\tilde{u}_h^{(2)})(\frac{\tilde{z}_h^{(2)}+ \tilde{z}}{2})+\\ \frac{1}{2}\rho^*(\tilde{u}_h^{(2)},\tilde{z}_h^{(2)})(\tilde{u}_h^{(2)}-\tilde{u})+ \tilde{\mathcal{R}}^{(3)(2)}.
			\end{equation}

			\subsection{Two-sided estimates for DWR using a saturation assumption}
			
			In the upcoming lemma, we derive two-sided bounds (efficiency and 
			reliablity) for the error estimator 
			$\tilde{\eta}^{(2)}$ defined by (\ref{Definition: Error Estimator}). 
			
			\begin{lemma}\label{Lemma:  Error Estimatorboundsremainder}
				If the assumptions of Theorem \ref{Theorem: Error Representation} are fulfilled, then
				the  computable error estimator $\tilde{\eta}^{(2)}$ 
				can be bounded from below and above as follows:
				\begin{equation}
				| J(u)-J(\tilde{u}) |-|J(u)-J(\tilde{u}_h^{(2)})| \leq |\tilde{\eta}^{(2)}| \leq | J(u)-J(\tilde{u}) |+ |J(u)-J(\tilde{u}_h^{(2)})|. \nonumber
				\end{equation}

				\begin{proof}
					We proof the bounds in the same way as in \cite{EnLaWi20}.
					We know that 
					$	|\eta| = |\tilde{\eta}^{(2)} - (\tilde{\eta }^{(2)} - \eta)| $,
					and therefore, we can conclude that 
					\begin{align*}
					|\eta| - |\eta - \tilde{\eta}^{(2)}| \leq |\tilde{\eta}^{(2)}| \leq 	|\eta|+|\eta - \tilde{\eta}^{(2)}|.
					\end{align*}
					The statement of the lemma follows with the identities,  $\eta-\tilde{\eta}^{(2)}=J(u)-J(\tilde{u})-J(\tilde{u}_h^{(2)})+J(\tilde{u}) = J(u)-J(\tilde{u}_h^{(2)})$ and 
					$\eta= J(u)-J(\tilde{u})$.
				\end{proof}
			\end{lemma}
			
			In the sequel, we impose a saturation assumption, which
			is a common assumption in hierarchical based error estimation \cite{BankWeiser1985,BoErKor1996,BankSmith1993,Verfuerth:1996a}. 
			Even if the solutions of the primal and adjoint problem in the enriched spaces are used,
			the saturation assumption may be violated as shown in \cite{BoErKor1996,EnLaWi20}. 
			
			However, for particular problems, quantities of interest and refinements, 
			it is possible to show the saturation assumption 
			\cite{DoerflerNochetto2002,Agouzal2002,AchAchAgou2004,FerrazOrtnerPraetorius2010,BankParsaniaSauter2013,CaGaGed16,ErathGanterPraetorius2018,Rossi2002}.  
			It heavily depends on  the quantity of interest, the finite dimensional spaces and the problem. 
			We impose the following assumption, 
			which is a slight generalization to \cite{EnLaWi20}.
			
			\begin{assumption}[Saturation assumption for the goal functional] \label{Assumption: Better approximation}
				Let $\tilde{u}_h^{(2)}$ be an arbitrary, but fixed approximation  in $U_h^{(2)}$, and let $\tilde{u}$ be some approximation in $U_h$. 
				Then we assume that 
				\[
				|J(u)-J(\tilde{u}_h^{(2)})| < b_h| J(u)-J(\tilde{u}) |
				\]
				for some  $b_h<b_0$ and some fixed $b_0 \in (0,1)$.
			\end{assumption}

			\begin{remark}
				For gradient based functionals like the flux, 
				one can use recovering techniques to reconstruct 
				the gradient as in \cite{KoPeRe03}. Here, under certain conditions, the saturation assumption can be shown.
			\end{remark}
			%
			\begin{remark} \label{Remark: pointevaluation}
				If $J(u)$ is just a point evaluation and the given point is a node in the mesh, then
				$|J(u)-J(\tilde{u}_h^{(2)})| =| J(u)-J(\tilde{u}) |$  
				in the case of higher-order interpolation as used in  \cite{BeRa01,RanVi2013,BaRa03}.
				Therefore, the saturation assumption is never fulfilled.
				If the given point is not on the grid, then 
				$|J(u)-J(\tilde{u}_h^{(2)})|$ converges to $| J(u)-J(\tilde{u}) |$ 
				provided that the mesh is locally refine around the evaluation point.
			\end{remark}
			\begin{theorem} \label{Theorem: Efficiency and Reliability with remainder}
				Let the Assumption~\ref{Assumption: Better approximation} be fulfilled. 
				Then the  computable error estimator $\tilde{\eta}^{(2)}$ is efficient and reliable, i.e. 
				\begin{equation*}
				\label {Estimate: hEffektivity+Remainder}
				\underline{c}_h|\tilde{\eta}^{(2)}| \leq | J(u)-J(\tilde{u}) | \leq 	\overline{c}_h|\tilde{\eta}^{(2)}|
				\quad \mbox{and} \quad
				\underline{c}|\tilde{\eta}^{(2)}| \leq | J(u)-J(\tilde{u}) | \leq 	\overline{c}|\tilde{\eta}^{(2)}|, 
				\end{equation*}
				where  $\underline{c}_h:= 1/(1+b_h)$, $\overline{c}_h:=1/( 1-b_h)$, $\underline{c}:= 1/(1+b_0)$,
				and $\overline{c}:=1/( 1-b_0)$. 
				%
				\begin{proof}
					The proof follows 
					from
					the estimates given in the proof of Lemma~\ref{Lemma:  Error Estimatorboundsremainder} and simple computations.  
					We follow  the same steps as in the proof of Theorem 3.3 in \cite{EnLaWi20}.
				\end{proof}
			\end{theorem}

			Following  \cite{RanVi2013,EnLaWi20}, we consider the practical discretization error estimator
			\begin{align} \label{Error Estimator: practical discretization wo Remainder}
			\tilde{\eta}_h^{(2)}:=\frac{1}{2}\rho(\tilde{u})(\tilde{z}_h^{(2)}-\tilde{z})+\frac{1}{2}\rho^*(\tilde{u},\tilde{z})(\tilde{u}_h^{(2)}-\tilde{u}) ,
			\end{align}
			%
			that corresponds
			to the theoretical discretization error estimator
			\begin{align*} \label{Error Estimator: theoretical discretization wo Remainder}
			\eta_h:=\frac{1}{2}\rho(\tilde{u})(z-\tilde{z})+\frac{1}{2}\rho^*(\tilde{u},\tilde{z})(u-\tilde{u}) .
			\end{align*}
			
			\begin{lemma} \label{Lemma: Difference Error Estimators}
				Let $\eta_h$ and $\tilde{\eta}_h^{(2)}$ be as defined above, 
				and let   $\tilde{u} \in U$ and $ \tilde{z} \in V$ be arbitrary but fixed.
				Furthermore, we assume that $\mathcal{A} \in \mathcal{C}^3(U,V)$ and $J \in \mathcal{C}^3(U,\mathbb{R})$.
				If $\tilde{u}_h^{(2)} \in U_h^{(2)}$ and $\tilde{z}_h^{(2)} \in V_h^{(2)}$ are some approximations  
				to  $u\in U$ and $z\in V$, respectively,
				then, for the approximations $\tilde{u}_h^{(2)}$ and $\tilde{z}_h^{(2)}$ from the enriched spaces $U_h^{(2)}$ and $V_h^{(2)}$,
				the following estimates
				\begin{equation}\label{Difference Error Estimators1}
				\begin{split}
				|J(u)-J(\tilde{u}_h^{(2)})|- |\mathcal{R}^{(3)} -\tilde{\mathcal{R}}^{(3)(2)} |- |\tilde{\eta}_{\tilde{z}_h^{(2)}}|- |\tilde{\eta}_{\tilde{u}_h^{(2)}}|\leq |\eta_h -\tilde{\eta}_h^{(2)}| ,\\
				|\eta_h -\tilde{\eta}_h^{(2)}| \leq |J(u)-J(\tilde{u}_h^{(2)})|+ |\mathcal{R}^{(3)} -\tilde{\mathcal{R}}^{(3)(2)} |+ |\tilde{\eta}_{\tilde{z}_h^{(2)}}|+ |\tilde{\eta}_{\tilde{u}_h^{(2)}}|,
				\end{split}
				\end{equation} 
				and 
				\begin{equation}\label{Difference Error Estimators2}
				|J(u)-J(\tilde{u})|- |\rho (\tilde{u})(\tilde{z})|-|\mathcal{R}^{(3)}| \leq |\eta_h| \leq |J(u)-J(\tilde{u})|+|\rho (\tilde{u})(\tilde{z})|+|\mathcal{R}^{(3)}|
				\end{equation} 
				hold. 
				Here, 
				$\mathcal{R}^{(3)}$ is defined in (\ref{Error Estimator: Remainderterm}), $\tilde{\mathcal{R}}^{(3)(2)}$ is from Theorem~\ref{Corrollary: discrete Error Representation}, 
				and we have
				\begin{equation*} 
				\tilde{\eta}_{\tilde{u}_h^{(2)}}:= -\rho(\tilde{u}_h^{(2)})(\frac{\tilde{z}_h^{(2)}+ \tilde{z}}{2}),  \qquad \tilde{\eta}_{\tilde{z}_h^{(2)}}:=  \frac{1}{2}\rho^*(\tilde{u}_h^{(2)},\tilde{z}_h^{(2)})(\tilde{u}_h^{(2)}-\tilde{u}).
				\end{equation*}
			\end{lemma}
			\begin{proof}
				The inequality (\ref{Difference Error Estimators2}) was already proven in \cite{EnLaWi20}.
				From Theorem \ref{Theorem: Error Representation} and Theorem~\ref{Corrollary: discrete Error Representation}, we get the identities 
				\begin{equation*}
				J(u)-J(\tilde{u})= \underbrace{\frac{1}{2}\rho(\tilde{u})(z-\tilde{z})+\frac{1}{2}\rho^*(\tilde{u},\tilde{z})(u-\tilde{u}) }_{\eta_h}
				+\rho (\tilde{u})(\tilde{z}) + \mathcal{R}^{(3)},
				\end{equation*}
				and
				\begin{equation*}
				J(\tilde{u}_h^{(2)})-J(\tilde{u})=\underbrace{ \frac{1}{2}\rho(\tilde{u})(\tilde{z}_h^{(2)}-\tilde{z})+\frac{1}{2}\rho^*(\tilde{u},\tilde{z})(\tilde{u}_h^{(2)}-\tilde{u}) }_{\tilde{\eta}_h^{(2)}}
				+\tilde{\eta}_{\tilde{u}_h^{(2)}} + \tilde{\eta}_{\tilde{z}_h^{(2)}}+\rho (\tilde{u})(\tilde{z}) + \tilde{\mathcal{R}}^{(3)(2)}.
				\end{equation*}					
				If we substract the equations from above, then we obtain  
				\begin{equation*}
				J(u)-	J(\tilde{u}_h^{(2)} )= \eta_h - {\tilde{\eta}_h^{(2)}} + \mathcal{R}^{(3)}-  \tilde{\mathcal{R}}^{(3)(2)} -\tilde{\eta}_{\tilde{u}_h^{(2)}} - \tilde{\eta}_{\tilde{z}_h^{(2)}}.
				\end{equation*}
				From this equality, we conclude (\ref{Difference Error Estimators1}) by using triangle inequality.
			\end{proof}
			
			\begin{remark}
				Indeed a refined analysis yields the inequalities
				\begin{equation*}
				|J(u)-J(\tilde{u}_h^{(2)})|- | \mathcal{R}^{(3)}-  \tilde{\mathcal{R}}^{(3)(2)} -\tilde{\eta}_{\tilde{u}_h^{(2)}} - \tilde{\eta}_{\tilde{z}_h^{(2)}}|\leq |\eta_h -\tilde{\eta}_h^{(2)}|  \leq |J(u)-J(\tilde{u}_h^{(2)})|+ | \mathcal{R}^{(3)}-  \tilde{\mathcal{R}}^{(3)(2)} -\tilde{\eta}_{\tilde{u}_h^{(2)}} - \tilde{\eta}_{\tilde{z}_h^{(2)}}|.
				\end{equation*} 
			\end{remark}
			
			Similar as in \cite{EnLaWi20}, we can now show the following result:
			\begin{lemma}\label{Lemma: Bounds for Error Estimator}
				If the conditions of Lemma~\ref{Lemma: Difference Error Estimators} are fulfilled, then the inequalities
				%
				\begin{align*}
				\label{Bounds for Error Estimator}
				|\tilde{\eta}_h^{(2)}| - \gamma(\mathcal{A},J,\tilde{u}_h^{(2)},\tilde{z}_h^{(2)},u,\tilde{u},\tilde{z}) \leq |J(u)-J(\tilde{u})| \leq |\tilde{\eta}_h^{(2)}| + \gamma(\mathcal{A},J,\tilde{u}_h^{(2)},\tilde{z}_h^{(2)},u,\tilde{u},\tilde{z}) 
				\end{align*}
				are valid, where 
				\begin{equation}
				\label{twick_Nov_13_2018_eq_1}
				\gamma(\mathcal{A},J,\tilde{u}_h^{(2)},\tilde{z}_h^{(2)},u,\tilde{u},\tilde{z}) :=|J(u)-J(\tilde{u}_h^{(2)})|+ |\mathcal{R}^{(3)} -\tilde{\mathcal{R}}^{(3)(2)} |+ |\tilde{\eta}_{\tilde{z}_h^{(2)}}|+ |\tilde{\eta}_{\tilde{u}_h^{(2)}}|+|\rho (\tilde{u})(\tilde{z})|+|\mathcal{R}^{(3)}|.
				\end{equation}
			\end{lemma}		
			\begin{proof}
				The proof follows the same idea as in the proof of Lemma 3.8 in \cite{EnLaWi20}.
			\end{proof}

			\subsection{Computable error estimator under a strengthend saturation assumption}
			
			Under a strengthened saturation assumption, the results from above can also be derived for the discretization error estimator $\tilde{\eta}_h^{(2)}$. 
			We suppose the following:
			%
			\begin{assumption}[Strengthened saturation assumption for the goal functional]  \label{Assumption: With Remainder}
				Let $\tilde{u}_h^{(2)}$ be an arbitrary, but fixed approximation  in $U_h^{(2)}$, 
				and let $\tilde{u}$ be some approximation. 
				Then the inequality
				\[
				\gamma(\mathcal{A},J,\tilde{u}_h^{(2)},\tilde{z}_h^{(2)},u,\tilde{u},\tilde{z}) < b_{h, \gamma}|
				J(u)-J(\tilde{u}) |,
				\]
				holds 
				for some $b_{h, \gamma}<b_{0, \gamma}$ with some fixed $b_{0, \gamma} \in (0,1)$,
				where $\gamma(\cdot)$ is defined in \eqref{twick_Nov_13_2018_eq_1}.
			\end{assumption}
			\begin{remark}
				Assumption~\ref{Assumption: With Remainder} implies Assumption~\ref{Assumption: Better approximation}. 
				However, if vice versa,  Assumption~\ref{Assumption: Better approximation} holds, then  Assumption~\ref{Assumption: With Remainder} is fulfilled up to higher-order terms  ($ |\mathcal{R}^{(3)} -\tilde{\mathcal{R}}^{(3)(2)} |$, $|\mathcal{R}^{(3)}|$), 
				and by the parts $|\rho (\tilde{u})(\tilde{z})|$, $\tilde{\eta}_{\tilde{u}_h^{(2)}}$, $\tilde{\eta}_{\tilde{z}_h^{(2)}}$ which can be controlled by the accuracy of the approximations.
				{In particular, if 
					the approximation $  \tilde{u}_h^{(2)}$ in the enriched space coincides with the finite element solution ${u}_h^{(2)}$, i.e. ${u}_h^{(2)} = \tilde{u}_h^{(2)}$, 
					then
					$\tilde{\eta}_{\tilde{u}_h^{(2)}}=0$. If the same condition, i.e. ${z}_h^{(2)} = \tilde{z}_h^{(2)}$, is fulfilled for the adjoint problem, then $\tilde{\eta}_{\tilde{z}_h^{(2)}}=0.$}
			\end{remark}

			\begin{theorem} \label{Theorem: Efficiency and Reliability without remainder}
				Let Assumption~\ref{Assumption: With Remainder} be satisfied. 
				Then the  practical error estimator $\tilde{\eta}_h^{(2)}$, 
				defined in \eqref{Error Estimator: practical discretization wo Remainder}, 
				is efficient and reliable, i.e.
				\begin{equation*}
				\label {Estimate: hEffektivity-Remainder}
				\underline{c}_{h, \gamma}|\tilde{\eta}_h^{(2)}| \leq | J(u)-J(\tilde{u}) | \leq 	\overline{c}_{h, \gamma}|\tilde{\eta}_h^{(2)}|
				\quad \mbox{and} \quad
				\underline{c}_{\gamma}|\tilde{\eta}_h^{(2)}| \leq | J(u)-J(\tilde{u}) | \leq 	\overline{c}_{\gamma}|\tilde{\eta}_h^{(2)}| , 
				\end{equation*}
				with the 
				positive 
				constants $\underline{c}_{h, \gamma}:= 1/(1+b_{h, \gamma})$, $\overline{c}_{h, \gamma}:=1/( 1-b_{h, \gamma})$, $\underline{c}_{\gamma}:= 1/(1+b_{0, \gamma})$, 
				$\overline{c}_{\gamma}:=1/( 1-b_{0, \gamma})$. 
				
				\begin{proof}
					The proof follows the same steps as in \cite{EnLaWi20}. The result can be derived from Lemma \ref{Lemma: Bounds for Error Estimator} and  Assumption~\ref{Assumption: With Remainder}. For further information, we refer to the proof in \cite{EnLaWi20}. 
				\end{proof}
			\end{theorem}

			\subsection{Bounds of the effectivity indices}
			
			As in \cite{EnLaWi20} we derive bounds for the effectivity indices $I_{eff}$ and $I_{eff,h}$ which are defined by
			\begin{equation} \label{Definition: Ieffs}
			I_{eff}:= \frac{|\tilde{\eta}^{(2)}|}{|J(u)-J(\tilde{u})|} 
			\quad \mbox{and} \quad
			I_{eff,h}:= \frac{|\tilde{\eta}_h^{(2)}|}{|J(u)-J(\tilde{u})|},
			\end{equation}
			respectively. The quantities $ \tilde{\eta}^{(2)}$ and $\tilde{\eta}_h^{(2)}$ are defined as in (\ref{Definition: Error Estimator}) and (\ref{Error Estimator: practical discretization wo Remainder}).
			We then obtain:
			\begin{theorem}[Bounds on the effectivity index] \label{Theorem: Ieffbounds}
				Let us  assume that $\mathcal{A} \in \mathcal{C}^3(U,V)$ and $J \in \mathcal{C}^3(U,\mathbb{R})$, and let   $\tilde{u} \in U$ and $ \tilde{z} \in V$ be arbitrary, but fixed. 
				Then the following two statements are true:
				\begin{enumerate}
					\item If Assumption~\ref{Assumption: Better approximation} is fulfilled,  then $I_{eff} \in [1-b_0,1+b_0]$, and	
					if additionally $b_h \rightarrow 0$, then  $I_{eff} \rightarrow 1$.
					\item If Assumption~\ref{Assumption: With Remainder} is fulfilled,  
					then $I_{eff,h} \in [1-b_{0,\gamma},1+b_{0,\gamma}]$, and  if additionally 
					$b_{h,\gamma} \rightarrow 0$, then  $I_{eff,h} \rightarrow 1$.
				\end{enumerate}
				
			\end{theorem}
			\begin{proof}
				The bounds immediately follow from 
				Lemma \ref{Lemma:  Error Estimatorboundsremainder} and Assumption~\ref{Assumption: Better approximation}, and  Lemma \ref{Lemma: Bounds for Error Estimator} and Assumption~\ref{Assumption: With Remainder}.
			\end{proof}
			
			\begin{remark}
				In practice, the efficiency indices
				$$	I_{eff,\mathcal{R}}:=\frac{|\tilde{\eta}_h^{(2)} - \rho(\tilde{u})(\tilde{z})+\tilde{\mathcal{R}}^{(3)(2)}|}{|J(u)-J(\tilde{u})|}$$ and 
				$I_{eff,h}$ 
				almost coincide.
			\end{remark}
			\begin{remark}
				In this section, we omit errors  coming from inexact data approximation and numerical quadrature.
			\end{remark}
			
			
			\section{Separation of the error estimator parts} 
			\label{Section: Localization and discussions on the error estimator parts} 
			In this section, we briefly describe the different 
			parts of the error estimator.
			The error estimator, which is derived in the previous section, consists of five parts. 
			The first three parts $ \tilde{\eta}_h^{(2)}$,  
			$\eta_k$, $\eta^{(2)}_\mathcal{R}$ 
			were already discussed in \cite{EnLaWi20}. 
			We will now focus on
			the fourth and fifth part, which are novel.
			For completeness of presentations, we give a short recap about the other parts.
			\begin{proposition}
				\label{prop_error_est}
				We split $\tilde{\eta}^{(2)}$ defined in (\ref{Definition: Error Estimator}) into the following five parts $\tilde{\eta}_h^{(2)}$, $\eta_k$, 
				$\eta^{(2)}_\mathcal{R}$, $\tilde{\eta}_{\tilde{u}_h^{(2)}}$, and $\tilde{\eta}_{\tilde{z}_h^{(2)}}$: 
				\begin{equation}
				\footnotesize
				\tilde{\eta}^{(2)}:=\underbrace{\frac{1}{2}\rho(\tilde{u})(\tilde{z}_h^{(2)}-\tilde{z})+\frac{1}{2}\rho^*(\tilde{u},\tilde{z})(\tilde{u}_h^{(2)}-\tilde{u}) }_{:= \tilde{\eta}_h^{(2)}}
				-\underbrace{\rho (\tilde{u})(\tilde{z})}_{:=\eta_k} + \underbrace{\tilde{\mathcal{R}}^{(3)(2)}}_{:=\eta^{(2)}_\mathcal{R}}\underbrace{-\rho({\tilde{u}_h^{(2)}})\left(\frac{\tilde{z}_h^{(2)}+\tilde{z}}{2}\right)}_ {:=\tilde{\eta}_{\tilde{u}_h^{(2)}}}+\underbrace{\rho^*(\tilde{u}_h^{(2)},\tilde{z}_h^{(2)})\left(\frac{\tilde{u}_h^{(2)}-\tilde{u}}{2}\right)}_{:=\tilde{\eta}_{\tilde{z}_h^{(2)}}}. \nonumber
				\end{equation}
				If the approximations ${u}_h^{(2)} \in U_h^{(2)}$ and $\tilde{z}_h^{(2)} \in V_h^{(2)}$ are the solutions of the primal and the adjoint problem, respectively, 
				then 
				$\tilde{\eta}^{(2)}$ coincides with the error estimator from \cite{EnLaWi20},  and the fourth and the fifth terms are zero.
			\end{proposition}
			
			\paragraph{The first  part $\tilde{\eta}_h^{(2)}$:} 
			
			Following
			\cite{RanVi2013,EnLaWi18,EnLaWi20}, $\tilde{\eta}_h^{(2)}$ is related to  the discretization error.
			We use the partition of unity approach developed in \cite{RiWi15_dwr} to localize $\tilde{\eta}_h^{(2)}$. Here the weak form of the estimator part can be used without 
			applying
			integration by part.

			\paragraph{The second  part $\eta_k$:}
			The  part $\eta_k=\rho(\tilde{u})(\tilde{z})$ measures the iteration error; 
			cf. \cite{EnLaWi20,RanVi2013,EnLaWi18}. 

			\paragraph{The third  part $\eta^{(2)}_\mathcal{R}$:}
			The third part $R^{(3)(2)}$ is of higher order, and  is often neglected in the literature. 
			In \cite{EnLaWi20}, we observed in several carefully designed studies 
			that this part is indeed neglectable.
			
			\paragraph{The fourth part $\tilde{\eta}_{\tilde{u}_h^{(2)}}$:}
			The forth part $\tilde{\eta}_{\tilde{u}_h^{(2)}}$ is a measure of the approximation quality of $\tilde{u}_h^{(2)}$.  If $\tilde{u}_h^{(2)}$ solves the (discrete)
			primal problem on the enriched space $U_h^{(2)}$, then we obtain that $\tilde{\eta}_{\tilde{u}_h^{(2)}}=0$. 
			The quantity indicates 
			whether
			the problem needs to be solved with higher accuracy
			or 
			the current approximation work sufficiently well.
			To this end, adaptive stopping criteria for both nonlinear and linear solvers can be designed 
			as 
			in \cite{EnLaWi20,EnLaWi18,EnLaNeiWoWi2020,RanVi2013,RaWeWo10,MeiRaVih109}.
			
			\paragraph{The fifth part $\tilde{\eta}_{\tilde{z}_h^{(2)}}$:}
			The fifth and last part $\tilde{\eta}_{\tilde{z}_h^{(2)}}$ provides 
			a quantity  to estimate 
			the approximation quality  of $\tilde{z}_h^{(2)}$.  In contrast to the fourth part, the fifth part $\tilde{\eta}_{\tilde{z}_h^{(2)}}=0$ if  $\tilde{z}_h^{(2)}$ solves the 
			(discrete)
			adjoint problem on the enriched space $V_h^{(2)}$. 
			The quantity indicates whether we should solve the problem more accurate or  it is fine to keep the current approximation. This can be used in an adaptive stopping rule criteria for the linear solver similar 
			to
			\cite{RaWeWo10,MeiRaVih109,ErnVohral2013}.

			\section{Algorithms}
			\label{Section: Algorithms}
			Based on the error estimator discussed in Proposition \ref{prop_error_est},
			we now design an adaptive algorithm. We would like to mention that this is just one realization of several classes of algorithms, 
			which can be constructed with this idea. Let us 
			start with the initial mesh $\mathcal{T}_h^1$ 
			and the corresponding finite element spaces 
			$V_h^1$,  $U_h^1$, $U_{h}^{1,(2)}$ and $V_{h}^{1,(2)}$, where $U_{h}^{1,(2)}$ and $V_{h}^{1,(2)}$ are the
			enriched finite element spaces. 
			For the resulting adaptively refined meshes   $\mathcal{T}_h^\ell$, with $\ell \geq 2$, we consider the following finite element spaces:
			$V_h^\ell$,  $U_h^\ell$, $U_{h}^{\ell,(2)}$ and $V_{h}^{\ell,(2)}$, where $U_{h}^{\ell ,(2)}$ and $V_{h}^{\ell,(2)}$ are the
			enriched finite element spaces. 
			To this end, we design Algorithm \ref{Outer DWR Algorithm}.

			\begin{algorithm}[H]
				\caption{}\label{Outer DWR Algorithm}
				\begin{algorithmic}[1]
					\State Start with some initial guess $u_h^{1}$, $\ell=1$, and some choice of $c_z$, $c_u \in \mathbb{R}$.\label{Algorithm: Start}
					\State Solve the primal problem: Find $u_h^\ell\in U_h^\ell$ such that $\mathcal{A}(u_h^\ell)(v_h^\ell)=0$ for all $v_h^\ell \in V_h^\ell$, using some nonlinear solver.  \label{Algorithm: small Primal}
					\State Solve the adjoint problem: Find $z_h^\ell\in V_h^\ell$ such that $\mathcal{A}'(u_h^\ell)(z_h^\ell,v_h^\ell)= J'(u_h^\ell)(v_h^\ell)$ for all $v_h^\ell \in U_h^\ell$, using some linear solver.\label{Algorithm: small Adjoint}
					\State Compute the interpolations  $u_h^{\ell,(2)}=I_{u}^{(2)}u_h^\ell\in U_h^{\ell,(2)}$  and $z_h^{\ell,(2)}=I_{z}^{(2)}z_h^\ell\in V_h^{\ell,(2)}$.\label{Algorithm: Interpolations}
					\State Compute the error estimators $\tilde{\eta}_h^{(2)},$ $\eta_k,$ $\eta^{(2)}_\mathcal{R}$, $\tilde{\eta}_{\tilde{z}_h^{(2)}}$, $\tilde{\eta}_{\tilde{u}_h^{(2)}}$ using $\tilde u_h^{(2)}=z_h^{\ell,(2)}$, $\tilde z_h^{(2)}=z_h^{\ell,(2)}$,   $\tilde u=u_h^{\ell}$, and $\tilde z=z_h^{\ell}$. \label{Algorithm: Compute_Estimators}
					\If{$|\tilde{\eta}_{\tilde{z}_h^{(2)}}|> c_z|\tilde{\eta}_h^{(2)}-\eta_k+\eta^{(2)}_\mathcal{R}|$}\label{Algorithm: Ifz}
					\State Solve the adjoint problem on the enriched spaces: Find $z_h^{\ell,(2)}\in V_h^{\ell,(2)}$ such that $\mathcal{A}'(u_h^{\ell,(2)})(z_h^{\ell,(2)},v_h^{\ell,(2)})= J'(u_h^{\ell,(2)})(v_h^{\ell,(2)})$ for all $v_h^{\ell,(2)} \in U_h^{\ell,(2)}$, using some linear solver. \label{Algorithm: Big Adjoint}
					\State  Go to Step \ref{Algorithm: Compute_Estimators}.
					\EndIf
					\If{$|\tilde{\eta}_{\tilde{u}_h^{(2)}}| > c_u|\tilde{\eta}_h^{(2)}-\eta_k+\eta^{(2)}_\mathcal{R}|$}\label{Algorithm: Ifu}
					\State  Solve the primal problem on the enriched spaces:  Find $u_h^{\ell,(2)}\in U_h^{\ell,(2)}$ such that $\mathcal{A}(u_h^{\ell,(2)})(v_h^{\ell,(2)})=~0$ for all $v_h^{\ell,(2)} \in V_h^{\ell,(2)}$, using some nonlinear solver. \label{Algorithm: Big Primal}
					\State  Go to Step \ref{Algorithm: Compute_Estimators}.
					\EndIf
					\If {$|\tilde{\eta}_h^{(2)}-\eta_k+\eta^{(2)}_\mathcal{R}|+|\tilde{\eta}_{\tilde{z}_h^{(2)}}|+|\tilde{\eta}_{\tilde{u}_h^{(2)}}|< TOL $}\label{Algorithm:  stopping Criteria}
					\State Algorithm terminates with final 
					\textbf{output} $J(u_h^\ell)$.
					\EndIf
					\State  Localize error estimator $\tilde{\eta_h^{(2)}}$, and $\eta^{(2)}_\mathcal{R}$ and mark elements. \label{Algorithm: Localization and marking}
					\State  Refine marked elements:$ \mathcal{T}_h^\ell \mapsto \mathcal{T}_h^{\ell+1}$, $\ell=\ell+1$. \label{Algorithm: Refinement}
					\State Go to Step~\ref{Algorithm: small Primal}
				\end{algorithmic}
			\end{algorithm}
			
			\begin{remark}
				We use the same interpolations as discussed in \cite{BeRa01,BaRa03}. For further information, we refer the reader to {\cite{BaRa03}; see pp. 43-44.}
			\end{remark}
			
			\begin{remark}
				In Step \ref{Algorithm: small Primal}, we use a Newton method with adaptive stopping rule using the estimator part $\eta_k$. The initial guess for the  Newton method was the solution on the previous grid. For further information about this Newton method we refer to \cite{EnLaWi20}. 
				The arising linear systems were solved {by means of} the direct solver UMFPACK \cite{UMFPACK}. 
				However,  iterative solvers could also be used, where the ideas from \cite{RanVi2013,RaWeWo10} can be exploited.
			\end{remark}
			\begin{remark}
				One can also use $|\tilde{\eta}_h|$ instead of $|\tilde{\eta}_h-\eta_k+\eta^{(2)}_\mathcal{R}|$.
				Indeed, in \cite{EnLaWi20}, it was observed that $|\eta^{(2)}_\mathcal{R}|$ is of higher-order, 
				and $\eta_k$ can be controlled by the choice of the accuracy of the solver. 
				Furthermore, it is sufficient to use the localized ${\tilde\eta_h^{(2)}}$ instead of ${\tilde\eta_h^{(2)}}$ and $\eta^{(2)}_\mathcal{R}$.
			\end{remark}
			\begin{remark}
				It is not required that the problems in Step \ref{Algorithm: small Primal} and Step \ref{Algorithm: small Adjoint} are solved accurate. 
				An estimate for  this error is $\eta_k$, which is perturbed by higher order terms.
			\end{remark}
			\begin{remark}
				In Step \ref{Algorithm: small Adjoint}, we use a Newton method with an adaptive stopping role 
				that is based on the estimator part $\eta_k$. 
				We take the solution from the previous grid as initial guess for the Newton iteration.
				We refer the reader to \cite{EnLaWi20} for further information about this Newton method.
			\end{remark}
			\begin{remark}
				If the problems in Step \ref{Algorithm: Big Primal} and Step \ref{Algorithm: Big Adjoint} are solved exactly, then $\tilde{\eta}_{\tilde{z}_h^{(2)}}=\tilde{\eta}_{\tilde{u}_h^{(2)}}=0$.  Therefore, the `if' conditions in Step \ref{Algorithm: Ifz} and \ref{Algorithm: Ifu} are false.
			\end{remark}
			\begin{remark}
				The localization and marking techniques in Step~\ref{Algorithm: Localization and marking} 
				coincide with those presented in \cite{EnLaWi20}.
				For more information on the localization, we refer to \cite{RiWi15_dwr}.
			\end{remark}
			\begin{remark}
				In Step \ref{Algorithm: Ifz} and \ref{Algorithm: Ifu}, we used the constants $c_u=c_z=0.5$. 
				In general, 
				one should choose these constants from the interval $(0,0.5]$.
			\end{remark}	
			\begin{remark}
				For the choices $c_u<0$ and $c_z<0$, the resulting algorithm coincides with the algorithm presented in \cite{EnLaWi20}. 
				Here the enriched problem needs to be solved at each level without 
				any interpolations.
				On the other hand, if we choose  $c_u=c_z=\infty$, then we never solve the enriched problem, and always use interpolations. 
				This leads to a similar approach as in \cite{RanVi2013}.
			\end{remark}

			\section{Numerical examples}
			\label{Section: Numerical examples}
			In this section, we discuss three different problems. 
			We also vary the goal functionals.
			More precisely, the first example deals with the Poisson equation
			and the average of the solution over the computational domain $\Omega$
			as simple linear model problem and quantity of interest, respectively.
			In the second test, we use a regularized $p$-Laplace equation, 
			and in the third example, we consider a stationary Navier-Stokes benchmark problem.
			The programming code is based on the finite element library deal.II 
			\cite{dealII90}.

			For the first two examples, we use continuous  bi-linear ($Q_1^c$) finite elements  for $V_h=U_h$, 
			and continuous bi-quadratic ($Q_2^c$)  finite elements for $V_h^{(2)}=U_h^{(2)}$ 
			in sense of Ciarlet \cite{Ciarlet:2002:FEM:581834}.  
			In the final example, we use the same configuration as in \cite{EnLaWi20}, i.e.,
			the finite element spaces $V_h=U_h$ and $V_h^{(2)}=U_h^{(2)}$ are based on 
			$ \left[Q_2^c\right]^2 \times Q_1^c$ and $\left[Q_4^c\right]^2 \times Q_2^c$ 
			finite elements, respectively.

			We use the following abbreviations for the error estimators used 
			in Algorithm \ref{Outer DWR Algorithm}:
			\textit{\textit{new}:} $c_u=c_z=0.5$,
			\textit{\textit{full}:}  $c_u=c_z=-1$,
			and \textit{\textit{int}:} $c_u=c_z=\infty$ (=$10^{100}$) in the numerical experiments.
			%
			%
			The choice $c_u=c_z=-1$ means that
			we always solve the primal and adjoint problems. 
			Therefore, for this case, the algorithm coincides with the algorithm presented in \cite{EnLaWi20} 
			(up to the starting point of the Newton iteration).  
			If  we have $c_u=c_z=\infty$, then this results in the case where we always use higher-order interpolation to the approximate $u-u_h$ and $z-z_h$ as done in \cite{BeRa01}.

\subsection{Poisson equation}
		
In the first example, we consider the Poisson equation  on the unit square $\Omega = (0,1)^2$. The problem formally reads as: Find $u \in H^1(\Omega)$ such that $-\Delta u = 1$ in $\Omega$ and  $u=0$ on $\partial \Omega$.
The exact solution is given by 
		\begin{align*}
			u(x,y)=\left(\frac{2}{\pi}\right)^4 \sum_{k=0}^{\infty}\sum_{l=0}^{\infty}\frac{sin\big((2k+1)\pi x\big)sin\big((2l+1)\pi y\big)}{(2k+1)(2l+1)\big((2k+1)^2+(2l+1)^2\big)}.
			\end{align*}
			The quantity of interest is given by 	$J(u)=\int_{\Omega} u \text{d}x $. The evaluation at the solution yields
					\begin{align*}
						J(u)=&\left(\frac{2}{\pi}\right)^6 \sum_{k=0}^{\infty}\sum_{l=0}^{\infty}\frac{1}{(2k+1)^2(2l+1)^2\big((2k+1)^2+(2l+1)^2\big)}\\
						=& \frac{1}{12} - \frac{31}{2\pi} \zeta(5) +\left(\frac{2}{\pi}\right)^5 \sum_{k=0}^{\infty}\frac{1}{(2k+1)^5\big(e^{(2k+1)\pi}+1\big)} \approx 0.03514425373878841,
					\end{align*}
		where $\zeta$ is the Riemann zeta function. 
				\begin{figure}[H]
					\centering			
					\scalebox{1.0}{\input{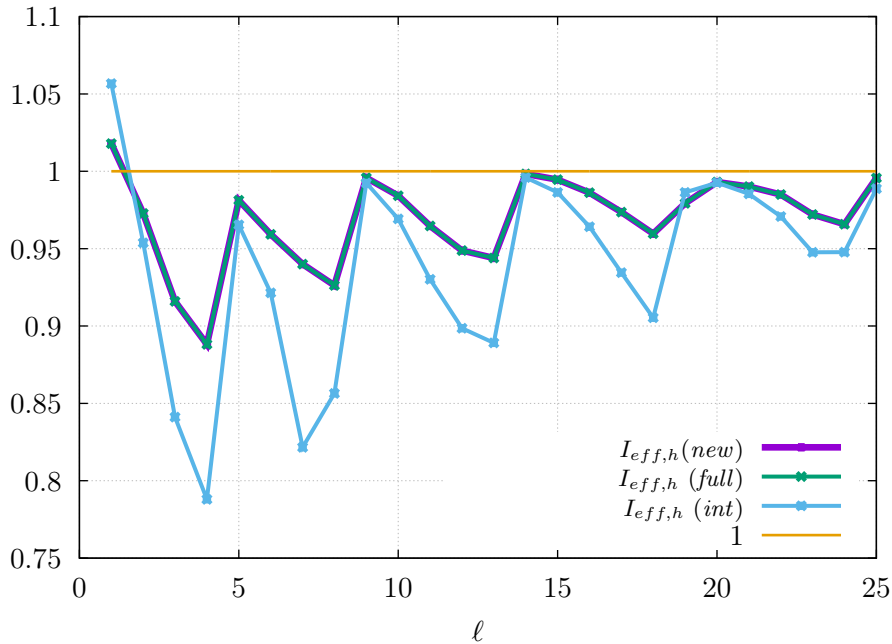}}
					\captionof{figure}{Effectivity indices for the Poisson equation. }
					\label{Figure: LaplaceIeff}
				\end{figure}
When
we compare the errors
in our quantity of interest for \textit{\textit{new}}, \textit{\textit{full}}, \textit{\textit{int}}, 
we observe that, for all three choices, we obtain almost the same error in comparison to the degrees of freedom (DOFs);
cf. Figure \ref{Figure: LaplaceErrors}.
Furthermore, we 
see from
Figure~\ref{Figure: Laplacesolves} that Algorithm~\ref{Outer DWR Algorithm} always decides to solve the primal problem on the enriched space on each level. However, the adjoint problem is never solved on the enriched space. 
If we compare the effectivity indices shown in Figure~\ref{Figure: LaplaceIeff}, then we observe that, for  \textit{\textit{new}} and \textit{\textit{full}}, the effectivity indices almost coincide. 
If we only use interpolation, then the result is slightly worse.
\begin{figure}[H] 
									\centering			
									\scalebox{1.0}{\input{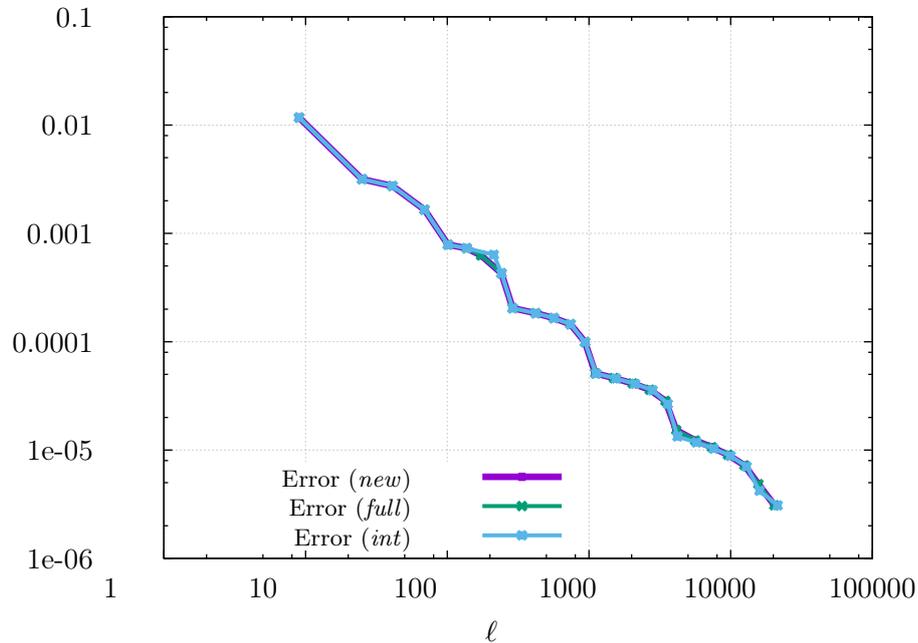}}
									\captionof{figure}{ Error vs DOFs for the Poisson equation.}\label{Figure: LaplaceErrors}
								\end{figure}

\begin{figure}[H]
				\centering			
				\scalebox{1.0}{\input{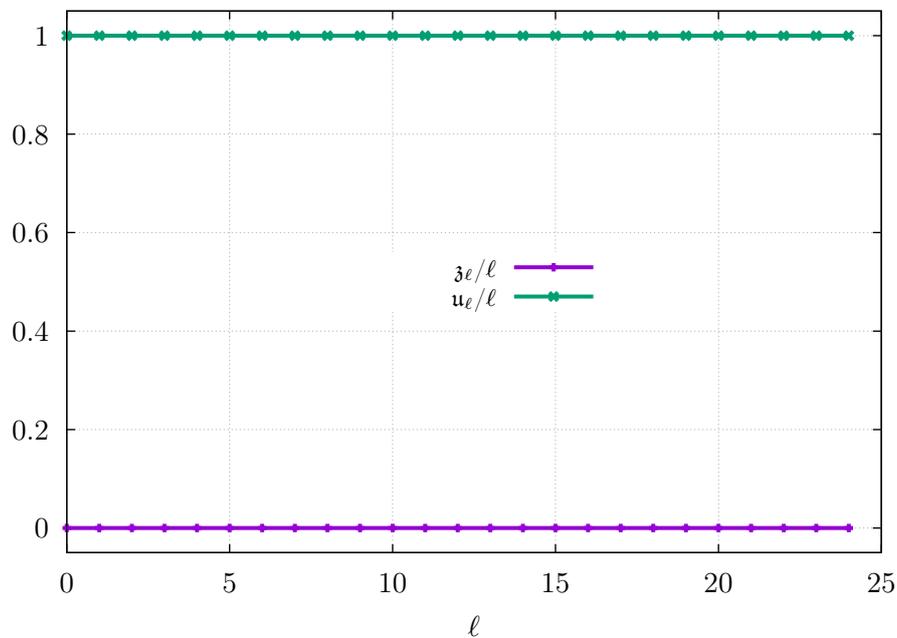}}
				\captionof{figure}{ Poisson equation: ratio of solves up to level $\ell$. Here, $\mathfrak{z }_\ell$ describes the number of enriched adjoint  solves needed up to level $\ell$,   and $\mathfrak{u }_\ell$ describes the number of enriched primal  solves needed up to level $\ell$}\label{Figure: Laplacesolves}
			\end{figure}

\subsection{Regularized p-Laplace equation}
In the second numerical example, we consider the regularized $p$-Laplace equation with  $\varepsilon =10^{-10}$ and $p=4$.  The computational domain $\Omega$ is a slit domain given by $\Omega = (-1,1)^2 \setminus \{0\} \times (-1,0)$ as visualized in Figure~\ref{Figure: Slit domain with BC}.
The given problem reads as: Find $u \in W^1_p(\Omega)$ such that 
\begin{equation*}
-\text{div}((|\nabla u |^2+ \varepsilon^2)^{\frac{p-2}{2}} \nabla u) =1 \; \text{in $\Omega$},
\quad \mbox{and}\;
u=0  \; \text{on $\Gamma_D$},
\quad
(|\nabla u |^2+ \varepsilon^2)^{\frac{p-2}{2}})\nabla u \cdot \vec{n}=0  \; \text{on $\Gamma_N$}.
\end{equation*}
The boundary conditions are visualized in the left subfigure of Figure~\ref{Figure: Slit domain with BC}. 
We impose Neumann 
and homogeneous Dirichlet boundary conditions 
on the left side and  on the right side of the slit, respectively.

In the right subfigure of Figure~\ref{Figure: Slit domain with BC}, a plot of the solution is given. Even for the $p=4$, similarities to the distance function, which is the first  eigenfunction of the $p$-Laplacian  for $p= \infty$,  described in \cite{KaHo2017a,FaBo2016a}, are visible. 
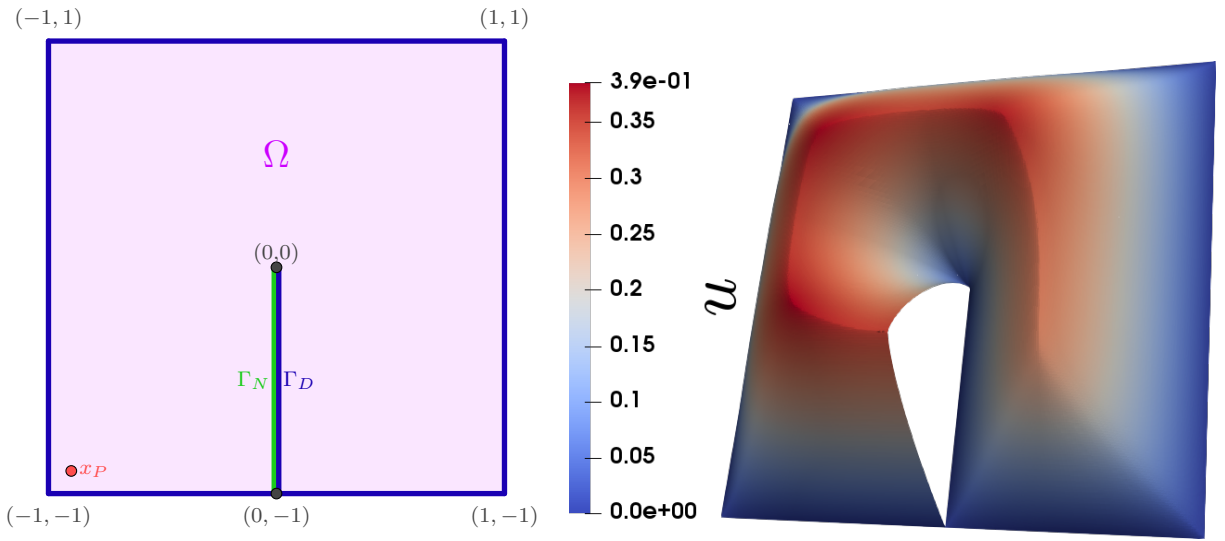
\begin{figure}[H]
\scalebox{1.0}{
\definecolor{uuuuuu}{rgb}{0.26666666666666666,0.26666666666666666,0.26666666666666666}
\definecolor{Farbe2}{rgb}{0.1,0,0.7}
\definecolor{Farbe1}{rgb}{0.1,0.8,0.1}
\definecolor{Farbe3}{rgb}{0.8,0.0,1}
\definecolor{Farbe4}{rgb}{1.0,0.3,0.3}
\definecolor{Farbe5}{rgb}{0.8,0.4,0.4}
\begin{tikzpicture}[line cap=round,line join=round,>=triangle 45,x=3cm,y=3cm]

\clip(-1.2,-1.2) rectangle (1.2,1.2);
\fill[line width=2pt,color=Farbe2,fill=Farbe3,fill opacity=0.10000000149011612] (1,-1) -- (1,1) -- (-1,1) -- (-1,-1) -- cycle;
\draw [line width=2pt,color=Farbe2] (1,-1)-- (1,1);
\draw [line width=2pt,color=Farbe2] (1,1)-- (-1,1);
\draw [line width=2pt,color=Farbe2] (-1,1)-- (-1,-1);
\draw [line width=2pt,color=Farbe2] (-1,-1)-- (1,-1);
\draw [line width=2pt,color=Farbe1] (-0.01,0)-- (-0.01,-1);
\draw [line width=2pt,color=Farbe2] (0.01,0)-- (0.01,-1);

\begin{scriptsize}
\draw [fill=uuuuuu] (0,-1) circle (2pt);
\draw [fill=Farbe4] (-0.9,-0.9) circle (2pt);
\draw[color=Farbe4](-0.8,-0.9)node {$x_P$};
\draw[color=uuuuuu] (0.0,-1.1) node {$(0,-1)$};
\draw[color=uuuuuu] (-1.0,-1.1) node {$(-1,-1)$};
\draw[color=uuuuuu] (1.0,-1.1) node {$(1,-1)$};
\draw[color=uuuuuu] (-1.0,1.1) node {$(-1,1)$};
\draw[color=uuuuuu] (1.0,1.1) node {$(1,1)$};
\draw [fill=uuuuuu] (0,0) circle (2pt);
\draw[color=uuuuuu] (0.0,0.06) node {(0,0)};
\draw[color=Farbe2] (0.1,-0.5) node {$\Gamma_D$};
\draw[color=Farbe1] (-0.1,-0.5)node {$\Gamma_N$};
\draw[color=Farbe3] (-0.0,0.5)node {$\text{\Large$\Omega$}$};
\end{scriptsize}
\end{tikzpicture}
}
{\vspace{-0cm}
\includegraphics[width=0.5\linewidth]{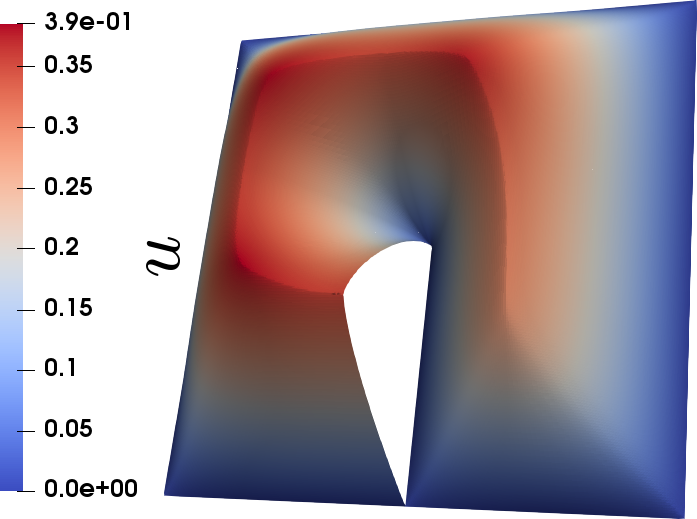}
}

\caption{The domain $\Omega$ with boundary conditions (left) and the solution of the problem (right). \label{Figure: Slit domain with BC}}
\end{figure}

\subsubsection{Integral evaluation} 
As 
first quantity of interest,  we again consider $J(u) = \int_{\Omega} u(x) dx \approx 0.71755$.
We
observe in Figure~\ref{Figure: Meal_PLaplace_Errors} that we obtain a similar error for either solving the adjoint and primal problem each time (Error(\textit{full})), for using the interpolation on each level (Error(\textit{int})), and for Algorithm~\ref{Outer DWR Algorithm} (Error(\textit{new})). 

As already noticed in \cite{EnLaWi20}, we observe higher-order convergence of the remainder term. The rate is approximately in the order of $\mathcal{O}(\text{DOFs}^{-\frac{3}{2}})$. For the errors and additionally the error estimator $\eta^{(2)}_h$, which is plotted for Algorithm~\ref{Outer DWR Algorithm}, the order of convergence is approximately $\mathcal{O}(\text{DOFs}^{-1})$.

In Figure~\ref{Figure: Meal_PLaplace_Solves}, the number of solves in the enriched space using  Algorithm~\ref{Outer DWR Algorithm} divided by the number of solves in the enriched space using 
the algorithm  given in \cite{EnLaWi20} is shown. 
We conclude that, on the first seven levels, the same solves as from \cite{EnLaWi20} are required. Then Algorithm~\ref{Outer DWR Algorithm} decides that we just have to solve either the primal or the adjoint problem on the enriched space, and use the interpolation in the other. After level $\ell =15$ only interpolation is used. 
Going back
to Figure~\ref{Figure: Meal_PLaplace_Errors}, we observe that, on the levels $\ell = 8-12$, the error for just using interpolation is slightly worse than 
for the other approaches. However, on finer levels this effect does not appear anymore. 
Excellent effectivity indices are observed as visualized in Figure~\ref{Figure: Meal_PLaplace_Ieff}.
For the full estimater, we observe almost no differences to the other  versions proposed.
In the case of Algorithm~\ref{Outer DWR Algorithm},
the efficiency index $I_{eff}$ is approximately equal to $1.25$ 
when using interpolation only.

If we compare the different meshes, which are visualized in Figure~\ref{Figure: Meal_PLaplace_Meshes}, 
then we 
detect
that, even after $28$ adaptive refinements, we end up in almost coinciding meshes.

\begin{figure}[H]
	\centering			
	\scalebox{1.0}{\input{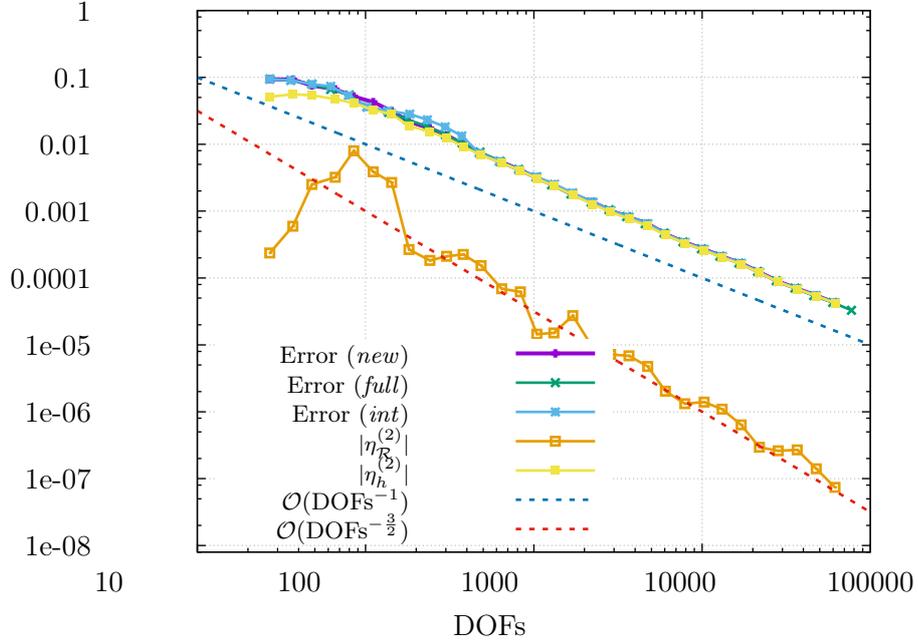}}
	\captionof{figure}{ $p$-Laplace for $p=4$, $\varepsilon=10^{-10}$: Integral evaluation: Error and error estimators  versus DOFs.}\label{Figure: Meal_PLaplace_Errors}
\end{figure}

\begin{figure}[H]
	\centering			
	\scalebox{1.0}{\input{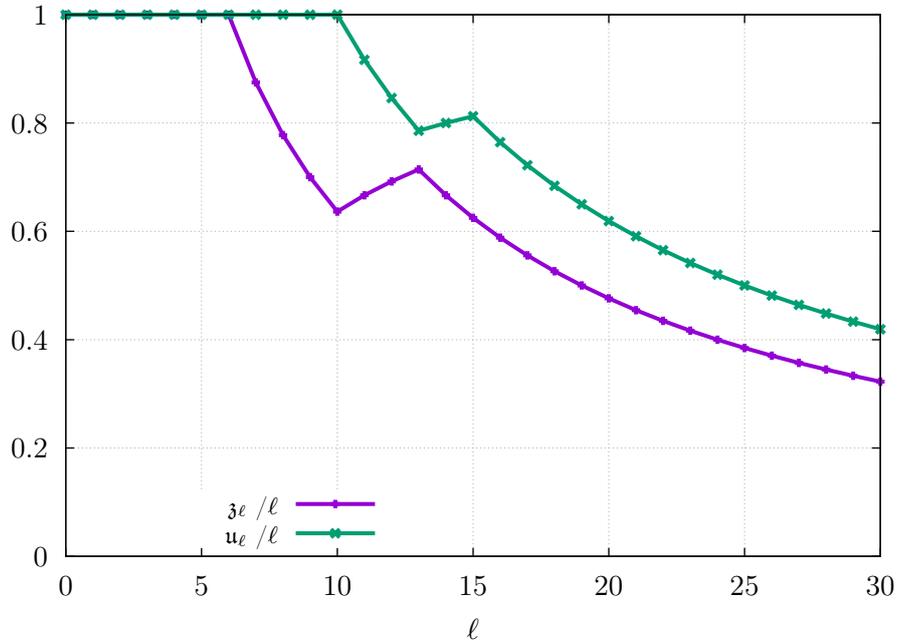}}
	\captionof{figure}{  $p$-Laplace for $p=4$, $\varepsilon=10^{-10}$: Integral evaluation:  ratio of solves up to level $\ell$. Here, $\mathfrak{z }_\ell$ describes the number of enriched adjoint  solves needed up to level $\ell$,   
	and $\mathfrak{u }_\ell$ describes the number of enriched primal  solves needed up to level $\ell$.}\label{Figure: Meal_PLaplace_Solves}
\end{figure}
\begin{figure}[H]
	\centering			
	\scalebox{1.0}{\input{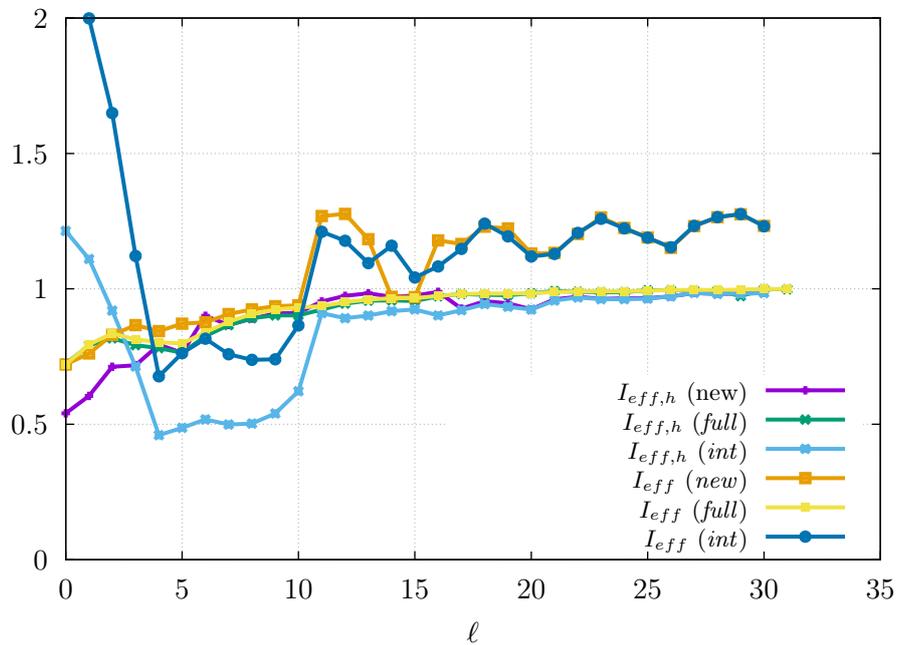}}
	\captionof{figure}{  $p$-Laplace for $p=4$, $\varepsilon=10^{-10}$: Integral evaluation: Effektivity idices.}\label{Figure: Meal_PLaplace_Ieff}
\end{figure}

\begin{figure}
\centering
\includegraphics[width=0.31\linewidth]{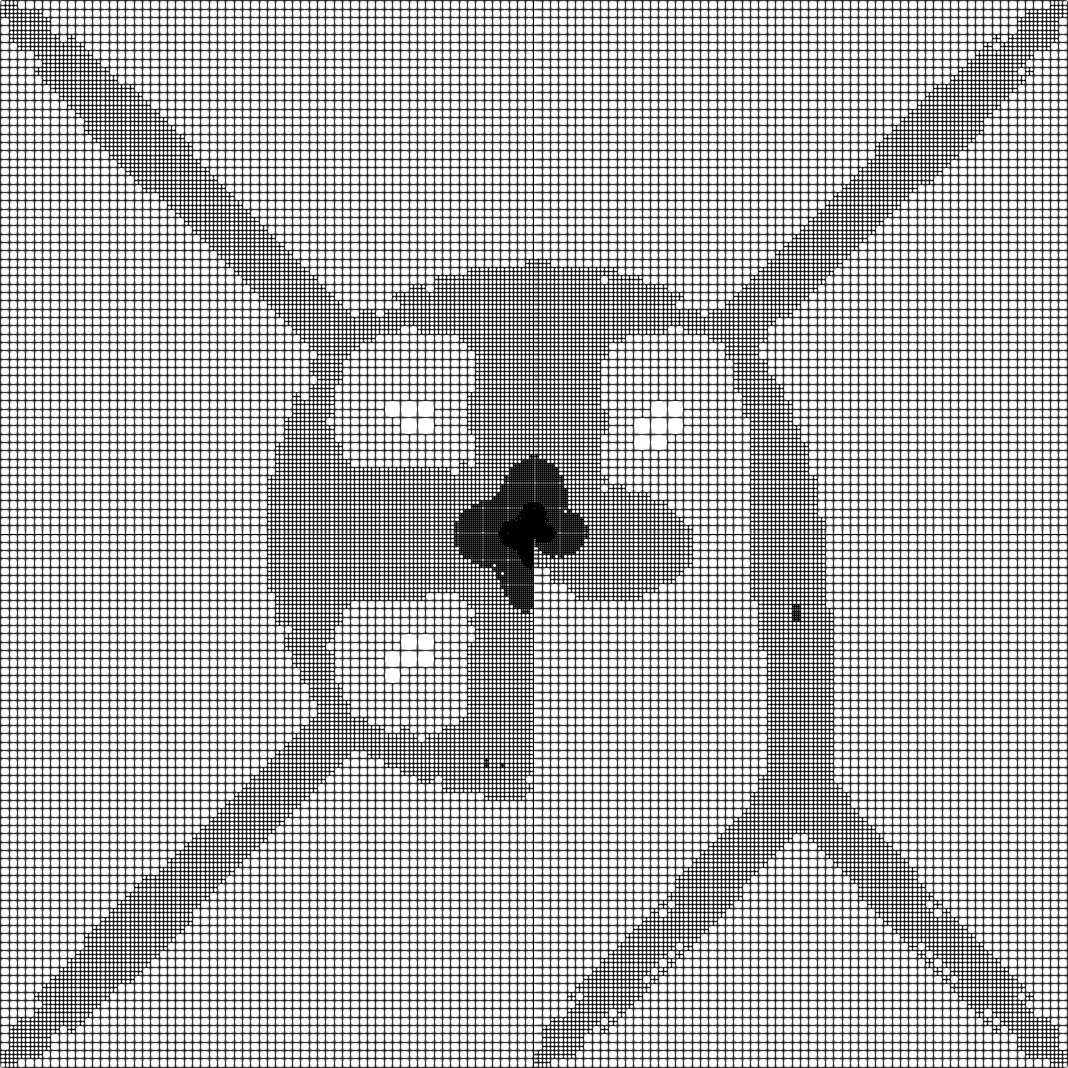}%
\hspace{0.035\linewidth}%
\includegraphics[width=0.31\linewidth]{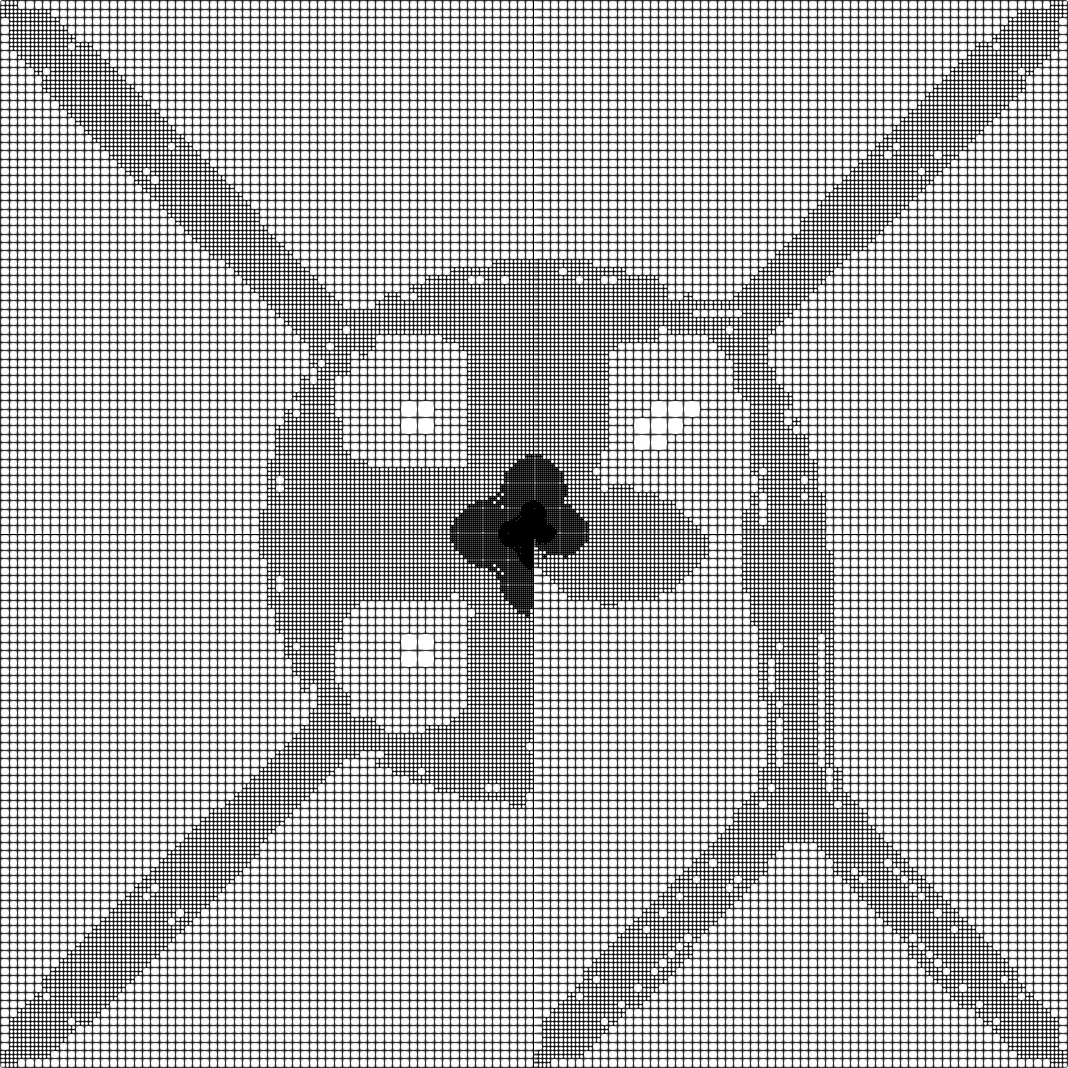}%
\hspace{0.035\linewidth}%
\includegraphics[width=0.31\linewidth]{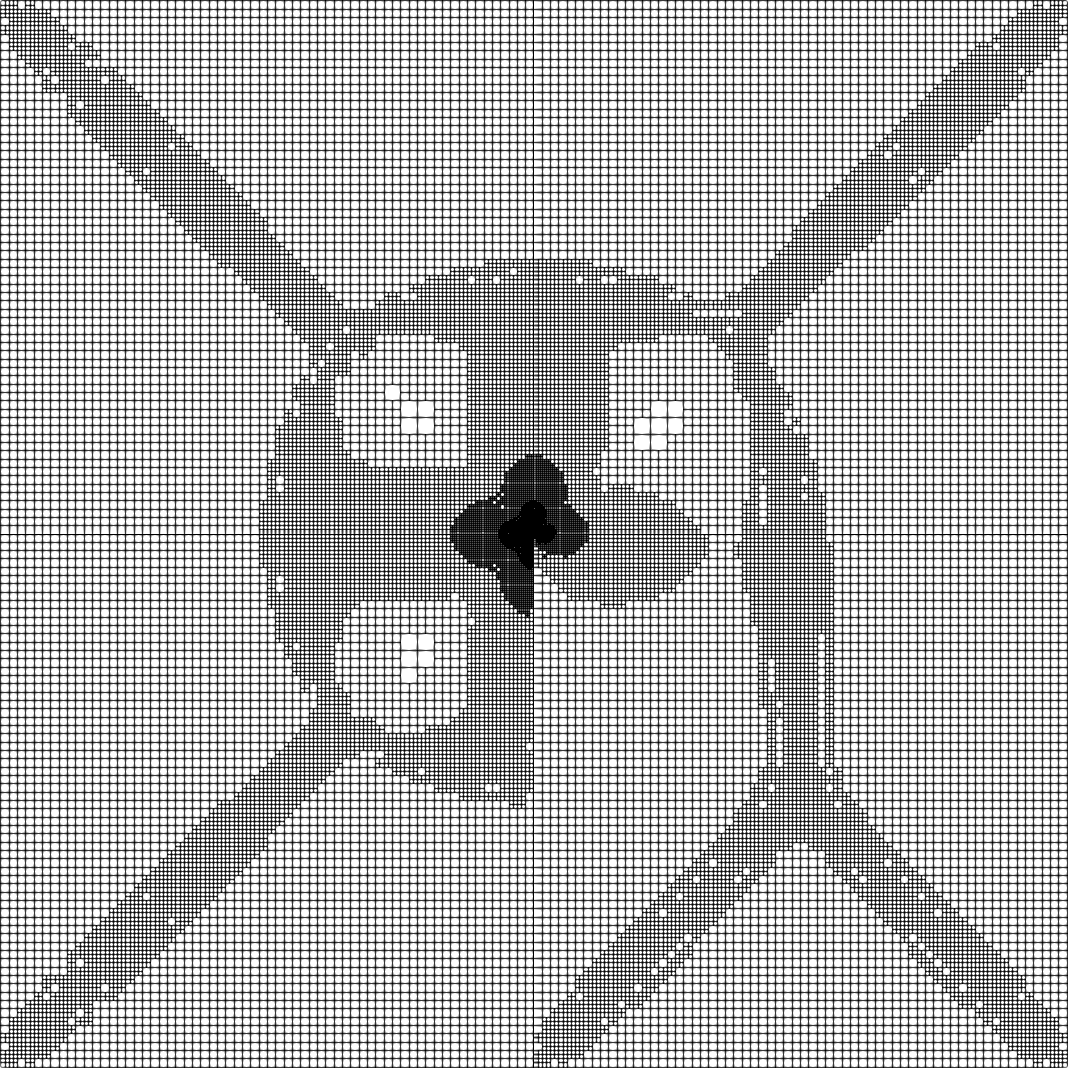}%
\caption{The resulting meshes on level $\ell=29$ for: solving always  the enriched spaces (left), using Algorithm~\ref{Outer DWR Algorithm} (middle), using interpolations on all levels (right).}
\label{Figure: Meal_PLaplace_Meshes}
\end{figure}

\subsubsection{Point evaluation} 
In this part, we consider the point evaluation $J(u) = u(x_P) \approx 0.04501097$ 
as quantity of interest,  where the point $x_P=-\frac{9}{10}(1,1)$ is 
also visualized in Figure~\ref{Figure: Slit domain with BC}.
Inspecting Table~\ref{Table: PLaplacePoint}, we observe that $\lim_{\ell \rightarrow \infty }I_{eff} =0$ 
due to the local refinement around the evaluation point as mentioned in Remark~\ref{Remark: pointevaluation}. 

Furthermore, 
Table~\ref{Table: PLaplacePoint}
shows that the effectivity indices  $I_{eff,h}$  and $I_{eff}$, defined in (\ref{Definition: Ieffs}), 
are both better for Algorithm~\ref{Outer DWR Algorithm} 
than
for using interpolation on every level. 
It is a bit surprising that the efficiency indices
$I_{eff,h}$ perform equally well.
Moreover, $\mathfrak{z }_\ell$   and  $\mathfrak{u}_\ell$ show that Algorithm~\ref{Outer DWR Algorithm} decides to solve the enriched problems on several levels.

In comparison with the algorithm proposed in our previous work \cite{EnLaWi20}, we save five times solving the primal problem on the enriched space, and 1 adjoint problem on the enriched space. 
In Figure~\ref{Figure: P-LaplacePointMesh24}, we observe that 
heavy refinement 
occurs
around our evaluation point. 

The position of the point was motivated by the singularity in the distance function, which is the first eigenfunction of the $p$-Laplacian  for $p= \infty$; see \cite{KaHo2017a,FaBo2016a}. 
	This singularity is also refined by our strategy, provided that it is sufficiently close to our point.
The errors in 
the
point evaluation are similar for interpolation and the new Algorithm~\ref{Outer DWR Algorithm}.

\begin{table}[H]
		\caption{$p$-Laplace for $p=4$, $\varepsilon=10^{-10}$: Point evaluation: In the first part of the table, the  levels are given. In the second part, we have the results for Algorithm \ref{Outer DWR Algorithm}, and in the third the results for using the interpolation on every level. $|V_h^{\ell}|$ gives the DOFs in the finite element space, and  $|V_h^{\ell,(2)}|$  gives the number of elements in the corresponding enriched space. The values $\mathfrak{z }_\ell$   and   $\mathfrak{u}_\ell$  describe the number of enriched solves required up to level $\ell$.  
			In the third part, $\mathfrak{z }_\ell$   and   \red{$\mathfrak{u}_\ell$}  are not 
			presented 
			since $\mathfrak{z }_\ell=\mathfrak{u}_\ell=0$.\label{Table: PLaplacePoint} }
	\scalebox{0.95}{
	\begin{tabular}{|l||r|r|l|l|l|l|l||r|r|l|l|l|}
		\hline
		&\multicolumn{7}{|c||}{\textit{new}}& \multicolumn{5}{|c|}{\textit{int}}\\ \hline
		$\ell$&  $|V_h^{\ell}|$  & $|V_h^{\ell,(2)}|$       & $I_{eff,h} $      &$I_{eff} $     &$\mathfrak{z }_\ell$   & $\mathfrak{u}_\ell$    & Error& $|V_h^{\ell}|$    & $|V_h^{\ell,(2)}|$      & $I_{eff,h}$    &$I_{eff} $ & Error      \\ \hline
		1  & 27    & 85    & 0.55  & 0.61  & 1  & 1  & 3.07E-02 & 27    & 85    & 0.18   & 0.29  & 3.07E-02 \\ \hline
		2  & 36    & 117   & 0.56  & 0.84  & 1  & 2  & 1.94E-02 & 36    & 118   & 0.48   & 0.78  & 1.95E-02 \\ \hline
		3  & 53    & 180   & 0.78  & 0.81  & 2  & 3  & 3.28E-03 & 53    & 180   & 1.12   & 1.41  & 3.28E-03 \\ \hline
		4  & 69    & 240   & 0.81  & 1.17  & 3  & 3  & 2.02E-03 & 69    & 240   & 0.55   & 1.17  & 2.02E-03 \\ \hline
		5  & 93    & 334   & 0.82  & 0.84  & 4  & 4  & 7.44E-04 & 93    & 334   & 0.85   & 1.21  & 7.44E-04 \\ \hline
		6  & 120   & 441   & 0.61  & 0.80  & 5  & 5  & 3.12E-04 & 123   & 453   & 1.75   & 0.66  & 3.25E-04 \\ \hline
		7  & 154   & 575   & 0.65  & 0.84  & 6  & 6  & 2.68E-04 & 167   & 626   & 1.33   & 0.45  & 1.55E-04 \\ \hline
		8  & 201   & 760   & 0.83  & 0.81  & 7  & 7  & 1.26E-04 & 215   & 815   & 0.58   & 0.45  & 1.55E-04 \\ \hline
		9  & 258   & 985   & 0.48  & 0.72  & 8  & 8  & 4.63E-05 & 275   & 1 051  & 0.78   & 0.02  & 1.09E-04 \\ \hline
		10 & 332   & 1 279  & 0.56  & 0.86  & 9  & 9  & 6.20E-05 & 356   & 1 370  & 0.69   & 0.03  & 6.28E-05 \\ \hline
		11 & 448   & 1 734  & 0.71  & 0.80  & 11 & 10 & 2.17E-05 & 468   & 1 810  & 0.78   & 0.06  & 2.86E-05 \\ \hline
		12 & 578   & 2 247  & 0.51  & 0.64  & 13 & 11 & 8.21E-06 & 602   & 2 334  & 0.61   & 0.02  & 3.04E-05 \\ \hline
		13 & 739   & 2 884  & 1.52  & 0.87  & 13 & 12 & 1.31E-05 & 790   & 3 082  & 0.71   & 0.04  & 1.59E-05 \\ \hline
		14 & 940   & 3 670  & 0.64  & 0.87  & 14 & 13 & 8.58E-06 & 1 014  & 3 973  & 1.18   & 0.07  & 8.74E-06 \\ \hline
		15 & 1 206  & 4 727  & 2.32  & 0.44  & 15 & 13 & 1.47E-06 & 1 300  & 5 109  & 0.54   & 0.12  & 5.60E-06 \\ \hline
		16 & 1 549  & 6 096  & 1.03  & 0.23  & 16 & 13 & 1.26E-06 & 1 701  & 6 696  & 1.76   & 0.17  & 1.45E-06 \\ \hline
		17 & 1 993  & 7 855  & 0.95  & 0.88  & 17 & 14 & 1.58E-06 & 2 196  & 8 658  & 2.52   & 0.09  & 1.58E-06 \\ \hline
		18 & 2 561  & 10 120 & 77.10 & 31.99 & 18 & 14 & 1.13E-08 & 2 831  & 11 181 & 1.52   & 0.09  & 1.60E-06 \\ \hline
		19 & 3 310  & 13 092 & 1.61  & 0.43  & 19 & 14 & 5.03E-07 & 3 642  & 14 419 & 3.94   & 0.41  & 4.01E-07 \\ \hline
		20 & 4 274  & 16 939 & 1.13  & 0.90  & 20 & 15 & 3.05E-07 & 4 728  & 18 734 & 922.58 & 54.95 & 1.16E-09 \\ \hline
		21 & 5 462  & 21 679 & 2.09  & 1.00  & 20 & 16 & 5.74E-07 & 6 089  & 24 164 & 2.98   & 0.03  & 2.72E-07 \\ \hline
		22 & 7 105  & 28 206 & 0.53  & 0.78 & 21 & 17 & 6.91E-08 & 7 913  & 31 433 & 3.16   & 0.02  & 2.01E-07 \\ \hline
		23 & 9 111  & 36 203 & 1.12  & 1.07  & 22 & 18 & 1.90E-07 & 10 237 & 40 698 & 2.93   & 0.02  & 1.57E-07 \\ \hline
		24 & 11 760 & 46 780 & 1.07  & 1.07  & 23 & 19 & 9.13E-08 & 13 214 & 52 569 & 2.08   & 0.05  & 1.53E-07 \\ \hline
		25 & 15 228 & 60 614 & 1.06  & 1.07  & 24 & 20 & 1.08E-07 & 17 122 & 68 164 & 1.41   & 0.02  & 1.80E-07 \\ \hline
	\end{tabular}
}

\end{table}

\begin{figure}
\centering
\includegraphics[width=1.0\linewidth]{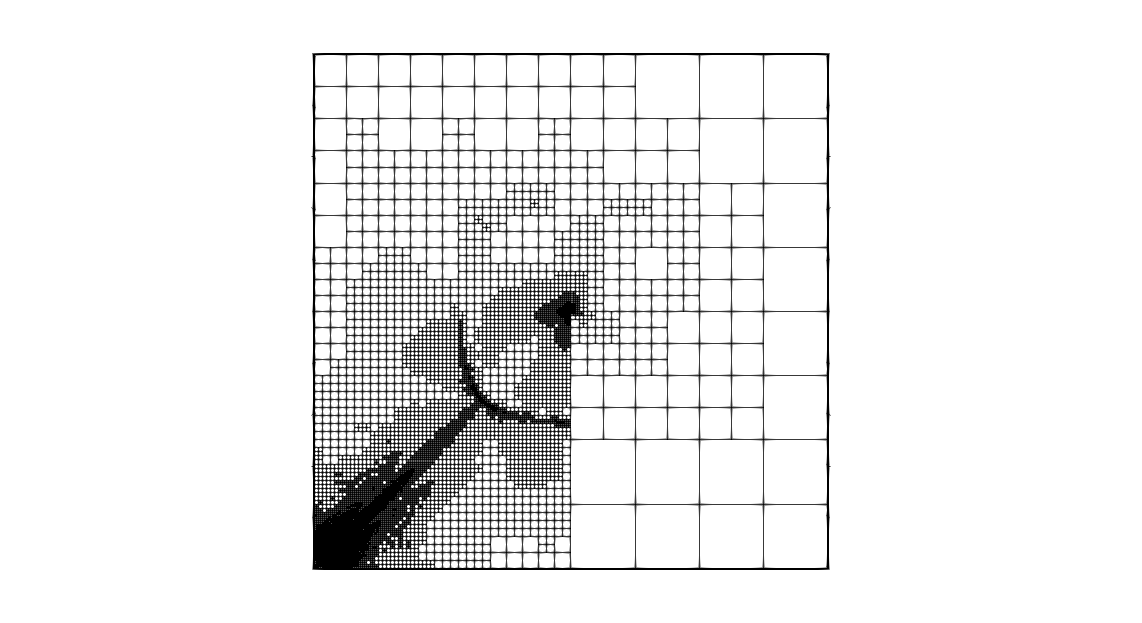}
\caption{$p$-Laplace for $p=4$, $\varepsilon=10^{-10}$: Point evaluation: Mesh on level $\ell =25$.}
\label{Figure: P-LaplacePointMesh24}
\end{figure}

\subsection{Navier-Stokes benchmark problem}
We now consider the stationary NS-benchmark problem NS2D-1\footnote{\url{http://www.featflow.de/en/benchmarks/cfdbenchmarking/flow/dfg_benchmark1_re20.html}}; see \cite{TurSchabenchmark1996}. 
{The computational domain $\Omega$ is given by}  
$(0,H) \times (0,2.2)\setminus \mathcal{B}$,
where $H =0.41$, and $\mathcal{B}:=B_\frac{1}{20}(0.2,0.2)$ 
is nothing but a
circle with center at $(0.2,0.2)$ and radius $\frac{1}{20}$. 
{The problem reads as follows:}
		Find $\textbf{u}:=(u,p) \in [H^1(\Omega)]^2 \times L^2(\Omega)$ such that
		\begin{align*}
		- \nu \Delta u + (u \cdot \nabla) u - \nabla p=& 0 \qquad  \qquad\text{in } \Omega,\nonumber\\ 
		\nabla \cdot u =&  0\qquad  \qquad\text{in } \Omega,\nonumber\\
		u=&0 \qquad  \qquad\text{on } \Gamma_{\text{no-slip}}, \\
		u=&\hat{u}\qquad  \qquad \text{on } \Gamma_{\text{inflow}}, \nonumber\\
		\nu \frac{\partial u}{\partial \vec{n}} - p \cdot \vec{n }=& 0 \qquad \qquad\text{on } \Gamma_{\text{outflow}}, \nonumber
		\end{align*}
		where $\nu = 10^{-3}$. The boundary parts  are given by 
		$\Gamma_{\text{outflow}} :=(\{x=2.2\} \cap \partial \Omega)\setminus \partial (\{x=2.2\} \cap \partial \Omega)$,
		$\Gamma_{\text{inflow}} :=\{x=0\} \cap \partial \Omega$,
		and
		$\Gamma_{\text{no-slip}} := \overline{\partial \Omega \setminus (\Gamma_{\text{inflow}} \cup \Gamma_{\text{outflow}})}$.

		The inflow is described by  $\hat{u}(x,y):=(3w(y)/10,0)$ with $w(y)=4y(H-y)/H^2$.
		The pressure is uniquely determined 
		due to the do-nothing condition
		prescribed on $\Gamma_{\text{outflow}}$; see \cite{HeRaTu96}.
		Our quantity of interest is given by the lift which is defined as 	\[
	J(\textbf{u}):= 500\int_{\partial \mathcal{B}} \left[\nu \frac{\partial u}{\partial \vec{n}} - p \vec{n }\right]\cdot \vec{e}_2\,\text{ d}s_{(x,y)},
\]
where $\vec{e}_2 = (0,1)$. 
The reference value $J(\textbf{u})=0.010618948146$ was taken from \cite{nabh1998high}.

In the numerical simulations, we observed that 
the `if' conditions (Step 6-7 and Step 9-10) 
in Algorithm \ref{Outer DWR Algorithm}
were entered, possibly multiple times, resulting in a significant 
improvement of the effectivity indices. In 
Table \ref{Table: NSAlgorithmusIeff}, these evaluations 
have the following correspondences to the previous algorithm:

\begin{enumerate}
\item Step 1 (Table \ref{Table: NSAlgorithmusIeff}) 
$\quad\widehat{=}\quad$ Step 4 (Alg. \ref{Outer DWR Algorithm}). For the computation of the estimators, we use $u_h^{\ell,(2)}=I_{u}^{(2)}u_h^\ell\in U_h^{\ell,(2)}$  and $z_h^{\ell,(2)}=I_{z}^{(2)}z_h^\ell\in V_h^{\ell,(2)}$.
\item Step 2 (Table \ref{Table: NSAlgorithmusIeff}) 
$\quad\widehat{=}\quad$ Step 6-7 (Alg. \ref{Outer DWR Algorithm}). For the computation of the estimators, we use $u_h^{\ell,(2)}=I_{u}^{(2)}u_h^\ell\in U_h^{\ell,(2)}$  and $z_h^{\ell,(2)}$ as the solution of the linear  problem: 
Find $z_h^{\ell,(2)} \in V_h^{\ell,(2)}$  such that $\mathcal{A}'(u_h^{\ell,(2)})(z_h^{\ell,(2)},v_h^{\ell,(2)})= J'(u_h^{\ell,(2)})(v_h^{\ell,(2)})$.
\item Step 3 (Table \ref{Table: NSAlgorithmusIeff}) 
$\quad\widehat{=}\quad$ Step 9-10 (Alg. \ref{Outer DWR Algorithm}). For the computation of the estimators, we use $u_h^{\ell,(2)}$ as the solution of {the non-linear problem}\: Find  $u_h^{\ell,(2)} \in  U_h^{\ell,(2)}$  such that $\mathcal{A}(u_h^{\ell,(2)})(v_h^{\ell,(2)})=~0$ and $z_h^{\ell,(2)}$ as in the previous executed step.
\item Step 4 (Table \ref{Table: NSAlgorithmusIeff}) 
$\quad\widehat{=}\quad$ Step 6-7 (Alg. \ref{Outer DWR Algorithm}). For the computation of the estimators, we use $u_h^{\ell,(2)}$ as the solution of {the non-linear problem}:  Find  $u_h^{\ell,(2)} \in  U_h^{\ell,(2)}$  such that $\mathcal{A}(u_h^{\ell,(2)})(v_h^{\ell,(2)})=~0$, and $z_h^{\ell,(2)}$ as the solution of {the linear problem}: Find $z_h^{\ell,(2)} \in V_h^{\ell,(2)}$  such that $\mathcal{A}'(u_h^{\ell,(2)})(z_h^{\ell,(2)},v_h^{\ell,(2)})= J'(u_h^{\ell,(2)})(v_h^{\ell,(2)})$.
\end{enumerate}

\begin{table}
		\caption{Navier-Stokes benchmark problem: Evolution of the effectivity indices during the exectuion of Algorithm \ref{Outer DWR Algorithm}. In Step 1, we use higher-order interpolation to approximate $z_h^{(2)}$ and $u_h^{(2)}$. 
			If Step 2 is executed, then we update $z_h^{(2)}$ by solving the adjoint problem where $u_h^{(2)}$ is the interpolation. 
			{After executing Step 3,}
			we replace the interpolation of $u_h^{(2)}$ by the solution of primal problem on the enriched space.  In Step 4,  $z_h^{(2)}$ and $u_h^{(2)}$ are the solutions of the primal and adjoint problem  on the enriched spaces, respectively. 
			{The sign {\color{red} X} means that}
			this step is not executed. 
			{ Lines \ref{Algorithm: Ifz} and \ref{Algorithm: Ifu} of Algorithm~\ref{Outer DWR Algorithm} 
				decide whether a step is execuded or not.}
			\label{Table: NSAlgorithmusIeff}}
	\scalebox{0.95}{
	\begin{tabular}{|c||c|c|c||c|c|c||c|c|c||c|c|c|}
		\hline
		 & \multicolumn{3}{|c|}{Step 1 (interpolation)} & \multicolumn{3}{|c|}{Step 2 (compute $z_h^{(2)}$)}& \multicolumn{3}{|c|}{Step 3  (compute $u_h^{(2)}$)}& \multicolumn{3}{|c|}{Step 4  (compute $z_h^{(2)}$)}\\ \hline
		$\ell$ &  $I_{eff,h}$   & $I_{eff,\mathcal{R}}$     & $I_{eff}$     &  $I_{eff,h}$   & $I_{eff,\mathcal{R}}$     & $I_{eff} $ &  $I_{eff,h}$   & $I_{eff,\mathcal{R}}$     & $I_{eff} $&  $I_{eff,h}$   & $I_{eff,\mathcal{R}}$     & $I_{eff} $   \\ \hline
		1           & 1.71   & 1.69  & 1.92  &{\color{red} X} & {\color{red} X}   & {\color{red} X}   & 2.23  & 2.2   & 1.55 & {\color{red} X}   & {\color{red} X}   & {\color{red} X}   \\ \hline
		2           &1.28   & 1.28   & 0.24 & 1.13                     & 1.13   & 0.24& 0.79  & 0.79  &0.78  & {\color{red} X}   & {\color{red} X}   & {\color{red} X}   \\ \hline
		3           &1.52   & 1.52 & 0.15&1.13                       & 1.13 & 1.15 &0.86  & 0.86    &0.85 & {\color{red} X}   & {\color{red} X}   & {\color{red} X}   \\ \hline
		4           &0.66   & 0.66 & 0.88 & 0.80                        & 0.80   & 0.88 & {\color{red} X}  & {\color{red} X}  & {\color{red} X} & {\color{red} X}   & {\color{red} X}   & {\color{red} X}   \\ \hline
		5           & 0.45   & 0.45  & 0.11 &{\color{red} X}               &{\color{red} X}  &{\color{red} X} & 0.72  & 0.72  & 1.00 & {\color{red} X}   & {\color{red} X}   & {\color{red} X}   \\ \hline
		6           &0.41   &0.41  & 0.31  & 0.88                     & 0.88 & 0.31 & 1.01     & 1.01 & 0.99 & {\color{red} X}   & {\color{red} X}   & {\color{red} X}   \\ \hline
		7           & 0.89  & 0.89 & 0.67 & {\color{red} X}                        & {\color{red} X}   & {\color{red} X}  & {\color{red} X}    & {\color{red} X}    & {\color{red} X}   & {\color{red} X}   & {\color{red} X}   & {\color{red} X}   \\ \hline
		8           & 0.41  & 0.45& 0.85 & 2.17                     & 2.21 & 0.85 & 1.54  & 1.58  & 1.60 & {\color{red} X}   & {\color{red} X}   & {\color{red} X}   \\ \hline
		9           & 2.81   & 2.81  & 2.86  & 0.71                       &0.71 & 2.86& 0.75  & 0.75  & 0.74 & {\color{red} X}   & {\color{red} X}   & {\color{red} X}   \\ \hline
		10          & 3.42   & 3.42     & 0.86  & 1.34                      & 1.34 & 0.86& 1.10 & 1.10  & 1.06 & {\color{red} X}   & {\color{red} X}   & {\color{red} X}   \\ \hline
		11          & 4.20   & 4.20  & 0.18 & 0.70                    &0.70 & 0.18 & 1.08  & 1.08  & 1.08 & {\color{red} X}   & {\color{red} X}   & {\color{red} X}   \\ \hline
		12          & 5.72   & 5.72 & 1.92  & 0.68                         &0.68& 1.92& 1.33  & 1.33  & 1.35 & {\color{red} X}  & {\color{red} X}  & {\color{red} X}  \\ \hline
		13          & 3.49   & 3.49 & 0.40  & 0.62                       & 0.62 & 0.40 & {\color{red} X}    & {\color{red} X}    & {\color{red} X}   & {\color{red} X}   & {\color{red} X}   & {\color{red} X}   \\ \hline
		14          & 5.11    &5.11   & 1.63  & {\color{red} X}                        &{\color{red} X}   & {\color{red} X}  &5.97  & 5.97    & 0.38  & 0.38  &0.38   & 0.38   \\ \hline
		15          &2.16   & 2.16 & 1.61  & {\color{red} X}                      & {\color{red} X}  & {\color{red} X}  & 1.99 & 1.99 & 0.94 & 0.94  & 0.94& 0.94   \\ \hline
		16          & 1.50  & 1.50  & 1.70  & 1.29                       & 1.29 & 1.70 & {\color{red} X}  & {\color{red} X}   & {\color{red} X}  & {\color{red} X}   & {\color{red} X}   & {\color{red} X}   \\ \hline
		17          & 6.25  & 6.24  & 1.95   & 1.30                     & 1.29 & 1.95  & 1.59  & 1.58  & 1.58 & {\color{red} X}   & {\color{red} X}   & {\color{red} X}   \\ \hline
		18          & 109.0   & 109.2  & 10.3 & 6.71                       & 6.88& 10.3 & 2.17   &2.34 &2.27   & {\color{red} X}   & {\color{red} X}   & {\color{red} X}   \\ \hline
		19          & 47.7& 47.7 &4.75 & 5.35                       & 5.35&4.75   & {\color{red} X}  & {\color{red} X}  & {\color{red} X} & {\color{red} X} & {\color{red} X}  & {\color{red} X}  \\ \hline
		20          & 231.0   &231.0  & 2.87  & {\color{red} X}                     &{\color{red} X} & {\color{red} X} & 226.4  & 226.4   &2.93  & 2.95 &2.93   &2.93 \\ \hline
		21          & 20.7  & 20.7 & 14.1     &{\color{red} X}                      & {\color{red} X} & {\color{red} X}     & 2.10  & 2.09 & 2.10& {\color{red} X}   & {\color{red} X}   & {\color{red} X}   \\ \hline
	\end{tabular}
}

\end{table}

Table~\ref{Table: NSAlgorithmusIeff} shows that we almost always improve the effectivity step by step. We notice that $I_{eff}$ does not depend on the choice of $z_h^{(2)}$ 
{since it coincides for Step 1 and Step 2 as well as for Step 3 and Step 4 provided that these steps are executed.}
Furthermore, we observe that, in certain meshes, 
{the use of}
interpolation leads to very bad  effectivity indices. 
However, Algorithm~\ref{Outer DWR Algorithm} improves this during the {adaptive} process, 
{although}
the saturation assumption is not fulfilled. The saturation assumption is also not fulfilled if $u_h^{(2)}$ is the exact discrete solution of the enriched space as shown in \cite{EnLaWi20}.

In Figure \ref{Figure: NSIeffhIeff}, we monitor a similar quality of the effectivity indices for \textit{new} and \textit{full}. 
Here we
observe  that the interpolation error estimator delivers a worse result. The resulting meshes for $\ell =20$ are shown in Figure \ref{fig:Full_New_Int}. Here also the meshes of \textit{new} and \textit{full} are more similar. 
This is not surprising since we perform more enriched solves.
For the error, which is discussed in Figure \ref{Figure: NSError},  
a particular conclusion could not be determined.

\begin{figure}
	\centering
	\includegraphics[width=0.95\linewidth]{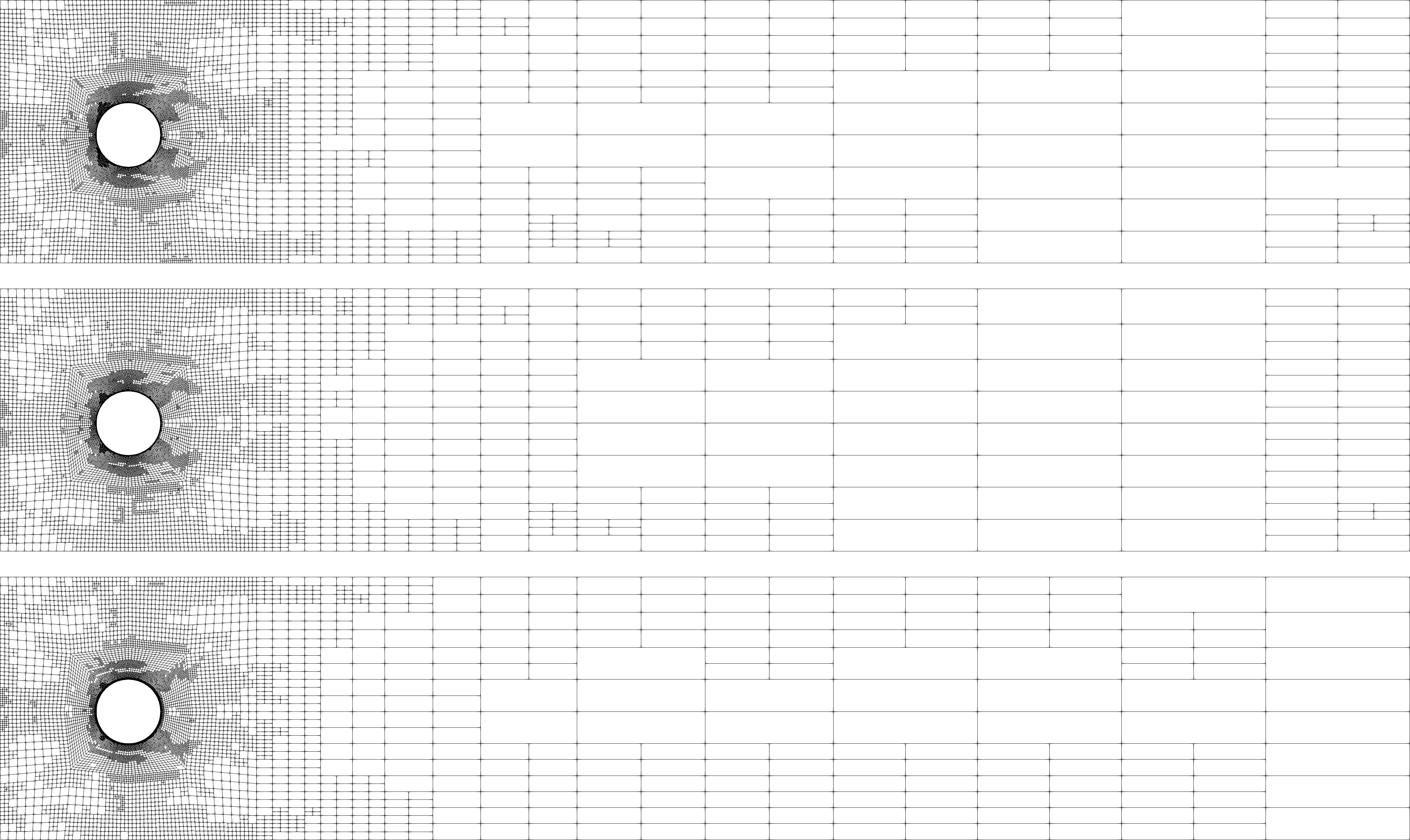}
	\caption{Navier-Stokes benchmark problem: The corresponding meshes for $\ell=20$ and  \textit{new}(top), \textit{full}(middle), \textit{int}(bottom).
}

	\label{fig:Full_New_Int}
\end{figure}
\begin{figure}[H]
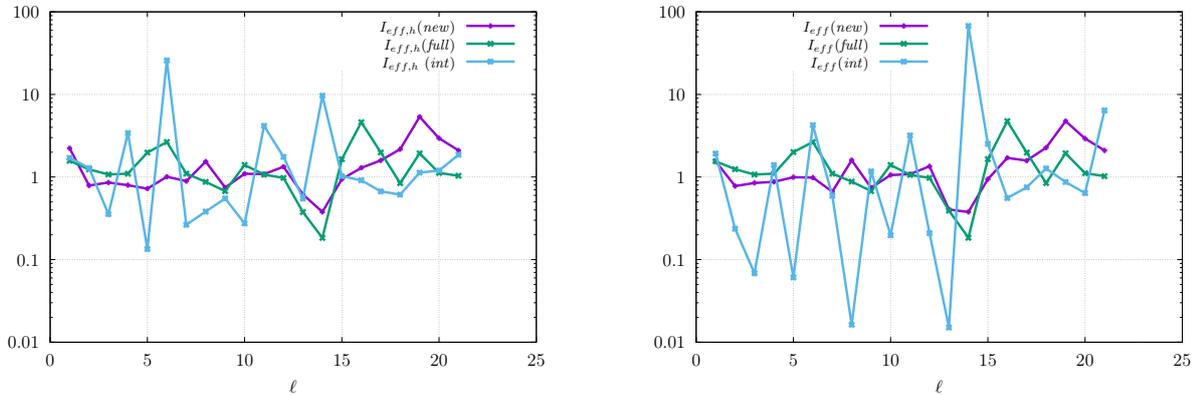


	\scalebox{0.61}{\input{Figures/Example1c_archive.tex}}\hfil	\scalebox{0.61}{\input{Figures/Example1c2_archive.tex}}
	\captionof{figure}{ Navier-Stokes benchmark problem: Effektivity indices.}\label{Figure: NSIeffhIeff}
\end{figure}

\begin{figure}[H]
	\centering			
	\scalebox{1.0}{\input{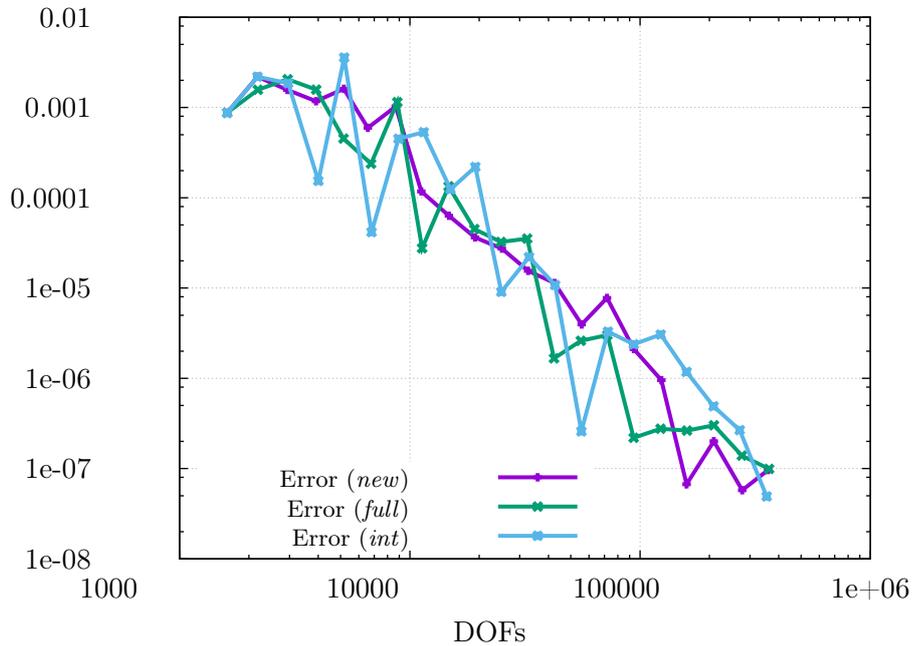}}
	\captionof{figure}{ Navier-Stokes benchmark problem: Errors versus DOFs}\label{Figure: NSError}
\end{figure}
\section{Conclusions}
\label{Section: Conclusions}
{We derived adaptive algorithms}
for
computationally attractive low-order 
finite elements and interpolations to realize 
goal-oriented a posteriori error estimation using the DWR approach.
Using saturation assumptions, we 
{rigorously proved}
two-side 
error estimates showing the efficiency and robustness. 
These findings 
were 
{supported by means of}
three numerical tests.
Therein, the newly suggested error estimator was compared 
to the full estimator and a version in which only interpolations 
are used. For linear problems (Example 1), all three variants coincide with respect the error behaviour. 
For nonlinear problems (Example 2), differences can be observed.
In the last numerical test (Example 3), a fluid-flow example 
was considered. Here, the PDE is semi-linear, but due to the 
convection term, the saturation assumption is not always fulfilled.
This could be observed in terms of bad effectivity indices every now 
and then. Moreover, in the last example, the mechanism of our proposed 
adaptive algorithm is highlighted because the switch from 
interpolations to enriched spaces in some iterations significantly 
improves the effectivity indices.
In future work, we plan to apply this algorithm to 
{
other applications, in particular, to multiphysics problems.}
		
   \section{Acknowledgments}
   This work has been supported by the Austrian Science Fund (FWF) under the grant
   P 29181
   `Goal-Oriented Error Control for Phase-Field Fracture Coupled to Multiphysics Problems'. 
Furthermore, the first two authors would like to thank  IfAM from the Leibniz Universt\"at Hannover (LUH) 
for the organization of their visit in Hannover in January 2020.  
The third author
would like to thank RICAM for his supported visit in Linz in November 2019, and
for funding from the
Deutsche Forschungsgemeinschaft (DFG, German Research Foundation)
under Germany's Excellence Strategy within the Cluster of
Excellence PhoenixD (EXC 2122).

		\bibliography{./lit}
		\bibliographystyle{abbrv}

	\end{document}